\documentclass{amsart}
\usepackage{amssymb,amsthm,amsmath}
\usepackage{enumerate}
\usepackage{cite}
\usepackage[hidelinks]{hyperref}
\usepackage{dsfont}
\usepackage{mathrsfs}

\def \Nats {\mathds{N}}
\def \Ints {\mathds{Z}}

\DeclareMathOperator{\restrictedto} {\upharpoonright}

\DeclareMathOperator{\Domain}{Dom}
\DeclareMathOperator{\Powerset}{P}
\DeclareMathOperator{\cl}{cl}
\DeclareMathOperator{\sscl}{\Lambda}

\newcommand{\set}[1]{\{#1\}}
\newcommand{\setarg}[2]{{\{#1\ |\ #2\}}}
\newcommand{\setcol}[2]{\{#1 : #2\}}

\def \strong {\leqslant}

\newtheorem{lemma}{Lemma}[subsection]
\newtheorem{theorem}[lemma]{Theorem}
\newtheorem{cor}[lemma]{Corollary}
\newtheorem{prop}[lemma]{Proposition}

\theoremstyle{definition}
\newtheorem{definition}[lemma]{Definition}
\newtheorem{fact}[lemma]{Fact}
\newtheorem{observation}[lemma]{Observation}
\newtheorem*{notation*}{Notation}
\newtheorem{notation}[lemma]{Notation}
\newtheorem{remark}[lemma]{Remark}

\newtheorem*{digression*}{Digression}

\DeclareMathOperator{\dm}{d}

\DeclareMathOperator{\Fin}{Fin}
\DeclareMathOperator{\fin}{fin}

\DeclareMathOperator{\tp}{tp}

\DeclareMathOperator{\arity}{a}
\newcommand{\pg}[1]{\mathcal{#1}}

\newcommand{\subfin}{\subseteq_{\fin}}

\newcommand{\flatsum}[1]{\Delta_{#1}}

\newcommand{\hg}[1]{\mathcal{#1}}

\newcommand{\subflat}{\trianglelefteq}
\newcommand{\psubflat}{\vartriangleleft}

\newcommand{\prerank}{\rho}
\newcommand{\OrdMod}{\mathfrak{O}}

\newcommand{\LEX}{\mathcal{L}_{\text{EX}}}

\title{The generic flat pregeometry}
\author{Omer Mermelstein}
\date{\today}
\address{Department of Mathematics, University of Wisconsin--Madison \\ 480 Lincoln Dr, Madison, WI 53706}
\email{omermelstein@wisc.edu}
\begin{document}
\begin{abstract}
	We examine the first order structure of pregeometries of structures built via Hrushovski constructions. In particular, developing matroid theoretic technology, we show that the class of flat pregeometries is an amalgamation class such that the pregeometry of the unbounded arity Hrushovski construction is precisely its generic.
	
	We show that the generic is saturated, provide an axiomatization for its theory, show that the theory is $\omega$-stable, and has quantifier-elimination down to boolean combinations of $\exists\forall$-formulas. We show that the pregeometries of the bounded-arity Hrushovski constructions satisfy the same theory, and that they in fact form an elementary chain.
\end{abstract}
\maketitle


\section{Introduction}

Via a study of flat pregeometries, we investigate the pregeometries associated to the hypergraph Hrushovski constructions --- the strongly minimal structures used to refute Zilber's conjecture. In particular, we show that these pregeometries are generic structures for classes of flat pregeometries, with the pregeometry associated to the unbounded arity Hrushovski construction being the generic for the class of all finite flat pregeometries. Using this characterization, we show that this pregeometry is saturated and $\omega$-stable of Morley rank $\omega^\omega$, and provide its theory and quantifier elimination down to boolean combinations of $\exists\forall$-formulas. Finally, we present the pregeometries of Hrushovski constructions of bounded arity as an elementary chain limiting to the unbounded arity pregeometry, showing that all are models of the same $\omega$-stable theory, differing in degree of saturation.

There are few known prototypical examples of $\omega$-stable theories, and even fewer which are not uncountably categorical. The canonical such examples are everywhere infinite forests, the generic structure for the amalgamation class of finite trees; Hrushovski's non-collapsed construction, which we define later in the introduction; and DCF$_0$, the theory of differentially closed fields of characteristic 0. These theories are all of Morley rank $\omega$, depicting the theory of the generic flat pregeometry as distinct. It is surprising to find any truly new $\omega$-stable theory and it is strong evidence that despite their reputation, the pregeometries of Hrushovski constructions are nice and natural objects. 

\medskip
A \emph{combinatorial pregeometry}, or matroid, is an abstract dependence relation on a set (see subsection \ref{subsection: pregeometry}). Classical examples are linear dependence in vector spaces, and algebraicity in field extensions. A pregeometry associated to a structure illustrates the degree of interaction between given elements --- e.g. linear independence indicates zero interaction, whereas a large set of a small linear dimension indicates many linear dependencies between the elements. In model theory, structures (types) that are sufficiently well behaved model theoretically have a naturally associated pregeometry called the \emph{forking} pregeometry. In the examples above, as well as many others, the forking pregeometry coincides with the classical dependence relation intrinsic to the structure.

A major program in model theory is to classify structures via their associated pregeometries. There are three types of pregeometries that arise ubiquitously in the model theoretic analysis of mathematical structures: disintegrated (set-like), locally modular (linear space-like), and field-like.

Zilber conjectured \cite{ZilberConjecture} that among the strongly minimal theories, those theories that are most model theoretically tame, these are the only types of pregeometries that arise. Hrushovski \cite{Hns} showed that the conjecture is false by producing a strongly minimal theory which was non-disintegrated yet interpreted no algebraic structures. Instead, his new strongly minimal theory is inherently combinatorial in nature. This was seen as the end for the hope of an orderly classification of all strongly minimal theories in terms of simple, understandable pregeometries. In this paper, we counter this position, giving an analysis of the pregeometry Hrushovski constructed, showing that it is in fact tame.

Where we can understand strongly minimal sets in terms of their pregeometries, many fruitful applications of model theory have been found. In many restricted cases -- see for example \cite{Rabinovich,StronglyMinimalReductsOfValuedFields, HrushovskiZilber-Zariskigeometries} -- the suggested trichotomy does in fact hold: Any strongly minimal subset of a tame enough structure must have one of these three sorts of pregeometries. This is fertile ground for interaction of model theory with other fields such as algebraic geometry, differentially closed fields, valuation theory, and many more -- see for example \cite{HrushovskiZilber-Zariskigeometries,Hrushovski-MordellLang,Zilber-ACurveAndItsAbstractJacobian}. In an attempt to salvage the trichotomy in a large class of cases, Hrushovski and Zilber \cite{HrushovskiZilber-Zariskigeometries} showed that the trichotomy holds for all Zariski structures. We hope that with our analysis of the pregeometries of structures arising from Hrushovski constructions, we can again pursue a general theory.

\vspace{8.4pt}
While the stated motivation is model theoretic, the paper reads as a study of \emph{flat} pregeometries (known as \emph{strict gammoids} or \emph{cotransversal matroids} to matroid theorists) with the intended purpose of constructing generic objects. Pregeometries and hypergraphs are the stars of the show, with model theory confined almost exclusively to subsection \ref{subsection: Model theory of G}.

To discuss the results of the paper, we must briefly survey Hrushovski's construction. For a finite hypergraph $\hg{A} = (M,R)$ ($R$ can be any set of finite subsets of $M$), we define the \emph{predimension} $\delta(\hg{A}) = |M| - |R|$ to be the difference between the number of vertices of $\hg{A}$ and the number of edges of $\hg{A}$. We write $\hg{A}\strong \hg{B}$ if there is no finite intermediate $\hg{A} \subseteq \hg{D}\subseteq \hg{B}$ such that $\delta(\hg{D})<\delta(\hg{A})$. In the context of a hypergraph $\hg{A} =(M,R)$, a finite subset $X\subseteq \hg{A}$ is deemed independent if $X\strong \hg{A}$ and there are no edges contained in $X$. In this way, we associate to each hypergraph a pregeometry $G_{\hg{A}}$, which the hypergraph $\hg{A}$ is said to \emph{represent}. These associated pregeometries were proven to have the property of \emph{flatness} by Hrushovski \cite{Hns}, which is key in showing Hrushovski's strongly minimal construction interprets no algebraic structures. The class $\mathcal{C}$ of all finite hypergraphs such that $\emptyset\strong \hg{A}$ is an amalgamation class under the notion of embedding $\strong$, and so has a unique countable generic structure $\mathcal{M}$. This is Hrushovski's (non-collapsed\footnote{Evans and Ferreira \cite{DavidMarcoTwo} showed that, under minor assumptions on the multiplicity function $\mu$, the pregeometry associated to Hrushovski's strongly minimal construction is identical to the one associated to the non-collapsed construction, which is easier to work with.}) construction for hypergraphs of unbounded arity. 

The pregeometry $G_{\mathcal{M}}$, denoted $\pg{G}$, is our main object of study in this paper. To allow detailed analysis of $\pg{G}$, we build it as the generic structure for $\mathfrak{C}$, the class of all finite flat pregeometries. We achieve this by following the Hrushovski construction outline: an amalgamation class with respect to a distinguished notion of embedding, paralleling the procedure for hypergraphs. However, with the amalgamation class of flat pregeometries lacking a predimension function or unique self-sufficient closures, the expected model theoretic properties cannot be automatically deduced by standard methods (e.g. Baldwin-Shi \cite{BaldwinShi}). We must first develop a deeper understanding of flat pregeometries.

\smallskip
In the first part of the paper, we introduce flatness and its accompanying notions $\sqsubseteq^*$, $\sqsubseteq$, $\prerank$ --- the purely geometric analogues of $\subseteq$, $\strong$, $\delta$, in the context of the amalgamation construction. Here we take a moment to clarify why standard methods fail us: Though the language we use for pregeometries is relational, not every set is embedded in a way that makes it into a ``substructure'' ($\sqsubseteq^*$-embedded), nor does the set generate a surrounding $\sqsubseteq^*$-embedded substructure in a unique way. As a result, the intersection of $\sqsubseteq^*$-embedded substructures is not necessarily $\sqsubseteq^*$-embedded --- pulling the rug from under submodularity, a critical property of $\delta$. In addition, unlike $\delta$ which takes integer values, $\prerank$ takes ordinal values $\leq \omega^{\omega+1}$, further complicating matters.

The contents of Section \ref{section: flatness} are the (strictly matroid theoretic) proofs that $\sqsubseteq^*$, $\sqsubseteq$, $\prerank$ are indeed adequate analogues, showing in Corollary \ref{corollary: sqsubseteq equivalence} that $H\sqsubseteq G$ is equivalent to $\prerank(H)$ being minimal among $H'\sqsubseteq^* G$ containing $H$, and to the existence of hypergraphs $\hg{A}\strong \hg{B}$ such that $G_{\hg{A}} = H$, $G_{\hg{B}} = G$. To the best of the author's knowledge, the notion of embedding $\sqsubseteq$ and the function $\prerank$ (which builds on the matroid theoretic $\alpha$-function) are new, yet intimately relevant to the study of strict gammoids. Combined with the Enumerative Construction (Definition \ref{definition: enumerative construction}) --- an explicit construction of a hypergraph representation of an arbitrary\footnote{Currently, the existence of a hypergraph representation is known only for \emph{finite} strict gammoids, via an application of Hall's Marriage Theorem, making the resulting hypergraph more opaque than using the Enumerative Construction. See \cite{EvansMatroid}.} strict gammoid --- these tools may be of independent interest to matroid theory. In subsection \ref{subsection: construction of G}, we use the characterization of $\sqsubseteq$ we in terms of $\strong$ to show that $(\mathfrak{C},\sqsubseteq)$ is an amalgamation class and that its generic structure is isomorphic to $\pg{G}$.

The construction echoes previous works by Evans \cite[Section 5]{EvansMatroid}, in which the author explores geometric characterizations of $\strong$ in the finite case, and Evans and Ferreira \cite[Section 6]{DavidMarcoOne}, in which the authors prove $\pg{G}$ has a weak form of genericity by applying a forgetful functor to the construction of $\mathcal{M}$. However, there, inability to depart from the hypergraph scaffolding leaves $\pg{G}$ impervious to further analysis. Using the technology developed in this paper to overhaul Evans and Ferreira's attempt yields what they were after \cite[Section 6, Problem]{DavidMarcoOne}, and opens $\pg{G}$ up to the model theoretic analysis they did not have the means to conduct.

\smallskip
In the second part of the paper, in subsection \ref{subsection: Model theory of G}, we use our detailed construction and characterization of $\pg{G}$ to analyze it model theoretically: we show it is saturated, we axiomatize its theory, and we show that its theory is $\omega$-stable and has quantifier elimination down to boolean combinations of a specified set of $\exists\forall$-formulas.

Finally, we develop a geometrical definition of arity in the flat setting. Having a thorough analysis of $\pg{G}$, we consider the pregeometries of the structures built by the bounded-arity Hrushovski constructions, as studied by Evans and Ferreira \cite{DavidMarcoOne}. Let $\mathcal{M}_n$ be Hrushovski's construction for hypergraphs of arity up to $n$. Let $\pg{G}_{n}$ be its pregeometry. Evans and Ferreira showed that $\hg{G}_n \not\cong \hg{G}_m$ whenever $n\neq m$. Moreover, they showed that even after localizing at finite sets, $\hg{G}_n$ and $\hg{G}_m$ remain non-isomorphic, seemingly demonstrating that there are $\omega$ many fundamentally different pregeometries arising from Hrushovski constructions. In contrast, we show that there are natural elementary embeddings: $\pg{G}_{3}\prec \pg{G}_{4}\prec \cdots \prec \pg{G}$. We conclude that not only is the theory of $\pg{G}$ $\omega$-stable, but it is the theory of all of the pregeometries of the canonical Hrushovski constructions, differing only by their level of saturation.

\section{Flatness}
\label{section: flatness}
\subsection{Pregeometry}
\label{subsection: pregeometry}

Write $A\subfin B$ to say that $A$ is a finite subset of $B$.

\begin{definition}
\label{definition: combinatorial pregeometry}
	A \emph{combinatorial pregeometry} $G$ is a set $X$ with a dimension function $\dm:\Fin(X)\to \Nats$ such that
	\begin{enumerate}
		\item
		$\dm(\emptyset) = 0$
		\item
		$\dm(A)\leq \dm(Ax)\leq \dm(A)+1$
		\item
		$\dm(A\cup B) + \dm(A\cap B) \leq \dm(A) + \dm(B)$
	\end{enumerate}
\end{definition}

Say that $A\subfin X$ is \emph{independent} if $\dm(A) = |A|$. Say that $Y\subseteq X$ is independent if every finite subset of $Y$ is independent. For $A\subfin X$ define $\cl(A) = \setarg{x\in X}{\dm(Ax) = \dm(A)}$. For $Y\subseteq X$ define $\cl(Y) = \bigcup_{Y_0\subfin Y} \cl(Y_0)$. For $Z\subseteq Y\subseteq X$, say that $Z$ is a basis for $Y$ if $Z$ is independent and $Y\subseteq\cl(Z)$. Say that $Y\subseteq X$ is closed if $\cl(Y) = Y$. Observe that an arbitrary intersection of closed sets is closed. We may extend $\dm$ to infinite subsets by taking $\dm(Y) = \sup\setcol{|Y_0|}{Y_0\subseteq Y\text{ is independent}}$, this definition coincides with $\dm$ on finite sets.

We interchangeably think of a pregeometry on a set $X$ as:
\begin{enumerate}
\item
A closure operator $\cl:P(X)\to P(X)$;
\item
A dimension function $\dm:P(X)\to \mathit{Card}$;
\item
A first order structure with relations $\setarg{I_n}{n\in \Nats}$ where $I_n\subseteq X^n$ is the set of independent $n$-tuples.
\item
A first order structure with relations $\setarg{D_n}{n\in \Nats}$ where $D_n\subseteq X^n$ is the set of dependent $n$-tuples.
\end{enumerate}

A subpregeometry $H\subseteq G$ is a substructure of a first order representation of $G$, and it is itself the pregeometry gotten by restricting the dimension function of $G$ to subsets of $H$. When no confusion arises, we may omit distinction between subsets and subpregeometries.

By convention, if several pregeometries $H,G,\dots$ are discussed simultaneously, we differentiate their dimension functions and closure operators with a subscript, i.e., $d_H, d_G,\cl_H, \cl_G$, etc. If all dimension functions in discussion are restrictions of some ambient dimension $\dm$, we omit the subscripts.

\subsection{Definition of flatness}

\begin{notation}
Let $G$ be a pregeometry and let $\Sigma = \set{E_1,\dots,E_k}$ be some ambient collection of closed subsets of $G$. For each non-empty set of indices $s\subseteq[k]$ we denote $E_s = \bigcap_{i\in s} E_i$, and for $s=\emptyset$ we denote $E_{\emptyset} = \bigcup_{i=1}^k E_i$. Denote also
\[
\flatsum{G}(\Sigma) = \sum_{\emptyset\neq s\subseteq [k]} (-1)^{|s|+1}\dm_G(E_s)
\]
\end{notation}

The alternating sum $\flatsum{G}(\Sigma)$ is the inclusion-exclusion principle, where the dimension function $\dm_G$ replaces cardinality. Like in inclusion-exclusion, the alternating sum should be thought of as reconstructing the dimension of $E_\emptyset$, the union of the sets $E_i$, based on the sum of the information found within each individual $E_i$. With this intuition in mind, it should make no difference whether we add to $\Sigma$ closed subsets of $E_i$, as we are adding no new information.

\begin{observation}
\label{observation: inclusion in flatsum}
In the context of a pregeometry $G$ and $\Sigma = \set{E_1,\dots, E_k}$, if there are some $i,j\leq k$ distinct with $E_i\subseteq E_j$, then $\flatsum{G}(\Sigma\setminus \set{E_i}) = \flatsum{G}(\Sigma)$.
\end{observation}

\begin{proof}
Observe that
\[
\flatsum{G}(\Sigma) - \flatsum{G}(\Sigma\setminus \set{E_i})
= \sum_{\emptyset\neq s\subseteq [k]\setminus\set{i}}(-1)^{|s|+1}\dm(E_{s\cup\set{i}})
\]
Note that $E_{s\cup\set{i,j}} = E_{s\cup\set{i}}$ for every $s\subseteq[k]\setminus\set{i,j}$. Since the dimensions of these sets appear with opposite signs in the sum, they cancel each other out.
\end{proof}

Cardinality is the simplest of dimension functions, but for an arbitrary dimension function there is no reason why $\flatsum{G}(\Sigma)$ should evaluate to the precise dimension of $E_{\emptyset}$\footnote{In fact, inclusion-exclusion always being correct is equivalent to a disintegrated dimension function}. Flatness is the statement that whenever we use an alternating sum to ``guess'' the dimension of a union, we may be overestimating, but never underestimating.

\begin{definition}
\label{def: flatness}
Say that a pregeometry $G$ is \emph{flat} if whenever $\Sigma$ is a finite collection of finite dimensional closed sets in $G$, then $\flatsum{G}(\Sigma) \geq \dm(\bigcup \Sigma)$.
\end{definition}

In the next observation and its corollary we see how adding information to $\Sigma$ changes the estimate $\flatsum{G}(\Sigma)$, given that $G$ is flat.

\begin{observation}
\label{observation: add set to Sigma}
Let $G$ be a flat pregeometry and let $\Sigma = \set{E_1,\dots,E_k}$ be a collection of finite dimensional closed sets in $G$. Let $X$ be a closed finite-dimensional subset of $G$ and denote $\Sigma_X = \setarg{E_i\cap X}{E_i\in\Sigma}$. Observe $\bigcup \Sigma_X = X\cap \bigcup \Sigma$. Then
\begin{align*}
\flatsum{G}(\Sigma\cup\set{X}) - \flatsum{G}(\Sigma) &= \dm_G(X) - \sum_{\emptyset\neq s\subseteq [k]} (-1)^{|s|+1} \dm_G(E_s\cap X)
\\
& = \dm_G(X) - \flatsum{X}(\Sigma_X)
\\
&\leq \dm_G(X) - \dm_G(X\cap \bigcup \Sigma)
\end{align*}
\end{observation}

\begin{cor}
\label{corollary: adding set to flatsum that is already contained in union}
If $G$ is flat, $\Sigma=\set{E_1,\dots,E_k}$ a collection of finite dimensional closed sets in $G$, and $X= \cl_G(Y)$ for some $Y\subseteq \bigcup \Sigma$, then $\flatsum{G}(\Sigma\cup\set{X})\leq \flatsum{G}(\Sigma)$.

Moreover, the inequality is strict if and only if $\dm_G(X) \neq \flatsum{X}(\Sigma_X)$, where $\Sigma_X = \setarg{E_i\cap X}{E_i\in\Sigma}$.
\end{cor}

\begin{notation}
Let $G$ be a pregeometry and let $\Sigma$ be a finite collection of finite dimensional closed sets in $G$. As above, we will often be interested in the dimensional information the closed sets of $\Sigma$ capture of some set other than $G$. We denote the relativization of $\Sigma$ with a subscript. If a subscript is already present, i.e.,  $\Sigma_G$ instead of $\Sigma$, replace it.
\\
For $X\subseteq G$ a subpregeometry of $G$, denote
\[
\Sigma_X = \setarg{E\cap X}{E\in\Sigma}.
\]
For $H\supset G$ a pregeometry containing $G$, denote
\[
\Sigma_H = \setarg{\cl_H(E)}{E\in\Sigma}.
\]
\end{notation}

\subsection{Distinguished embeddings}

In observation \ref{observation: add set to Sigma}, $\Sigma_X$ is the restriction of the elements of $\Sigma$ to $X$. For an arbitrary $X$, the value of $\flatsum{X}(\Sigma_X)$ may differ from $\flatsum{G}(\Sigma)$. One obvious reason is that intersecting $\bigcup\Sigma$ with a smaller set may result in a drop in dimension. Bar that, a subtler possibility is that the dimension of intersections between the elements of $\Sigma$ is not witnessed in full in $X$.

Given our intuition regarding inclusion-exclusion, if $\dm(\bigcup \Sigma_X) = \dm(\bigcup \Sigma)$, we should have $\flatsum{X}(\Sigma_X) \geq \flatsum{G}(\Sigma)$, as the restrictions to $X$ hold less information than the unrestricted sets in $G$. However, this non-witnessing of intersections may result in $\flatsum{X}(\Sigma_X)$ being strictly smaller than $\flatsum{G}(\Sigma)$. Indeed, this would imply that the pregeometry on $X$ is displaying non-flat behavior This motivates the next definition.

\begin{definition}
For pregeometries $H\subseteq G$ write $H\sqsubseteq^* G$ if whenever $X_1,X_2\subseteq H$ are closed in $H$ then
\[
\dm_H(X_1\cap X_2) = \dm_G(\cl_G(X_1)\cap \cl_G(X_2)).
\]
\end{definition}

\begin{observation}
The relation $\sqsubseteq^*$ is transitive.
\end{observation}

\begin{observation}
Assume $F\sqsubseteq^* G$ and $F\subseteq H\subseteq G$. Then $F\sqsubseteq^* H$.
\end{observation}

\begin{lemma}
\label{lemma: in sqsubseteq^* the intersection and closure commute}
If $H\sqsubseteq^* G$ and $E_1,\dots, E_k$ are closed subsets in $H$, then
\[
\cl_G\left(\bigcap_{i=1}^k E_i\right) = \bigcap_{i=1}^k \cl_G(E_i)
\]

\end{lemma}

\begin{proof}
We prove by induction on $k$. If $k=1$, then there is nothing to show. Otherwise, by $H\sqsubseteq^* G$ and induction hypothesis,
\[
\dm_H\left(E_1\cap \bigcap_{i=2}^k E_i\right) = \dm_G\left(\cl_G(E_1)\cap \cl_G\left(\bigcap_{i=2}^k E_i\right)\right) = \dm_G\left(\bigcap_{i=1}^k \cl_G\left(E_i\right)\right).
\]
As $\displaystyle{\bigcap_{i=1}^k E_i \subseteq \bigcap_{i=1}^k \cl_G(E_i)}$, the desired equality is evident.
\end{proof}

\begin{cor}
\label{corollary: sqsubseteq^* plays nicely with sums}
Let $H\sqsubseteq^* G$ and let $\Sigma$ be a finite set of finite dimensional closed subsets of $H$. Then
\[
\flatsum{G}(\Sigma_G) = \flatsum{H}(\Sigma)
\]
\end{cor}

\begin{proof}
Denoting $\Sigma = \set{E_1,\dots, E_k}$ and $\Sigma_G = \set{F_1,\dots, F_k}$, where $F_i = \cl_G(E_i)$, Lemma \ref{lemma: in sqsubseteq^* the intersection and closure commute} gives  $\dm_H(E_s) = \dm_G(F_s)$ for every non-empty $s\subseteq[k]$. Therefore, the sum $\flatsum{G}(\Sigma_G)$ is precisely the sum $\flatsum{H}(\Sigma)$.
\end{proof}

\begin{cor}
\label{cor: sqsubseteq^* in flat implies flatness}
If $H\sqsubseteq^* G$ and $G$ is flat, then also $H$ is flat.
\end{cor}

\begin{proof}
Let $\Sigma$ be a finite collection of finite dimensional closed sets in $H$. By monotonicity of dimension, flatness of $G$, and $H\sqsubseteq^* G$,
\[
\dm(\bigcup\Sigma) \leq \dm(\bigcup\Sigma_G) \leq \flatsum{G}(\Sigma_G) = \flatsum{H}(\Sigma)
\]
\end{proof}

The next definition is a significant strengthening of $\sqsubseteq^*$ that will later allow amalgamation of flat pregeometries.

\begin{definition}
\label{def: geometric self-sufficiency}
For pregeometries $H\subseteq G$, write $H\sqsubseteq G$ if whenever $\Sigma$ is a finite collection of finite dimensional closed sets in $G$, then
\[
\flatsum{H}(\Sigma_H) \leq \flatsum{G}(\Sigma).
\]
Say that $H$ is \emph{strongly embedded} in $G$.
\end{definition}

\begin{observation}
\label{observation: sqsubseteq implies sqsubseteq^*}
$H\sqsubseteq G \implies H\sqsubseteq^* G$.
\end{observation}

\begin{proof}
Let $X_1,X_2$ be closed subsets of $H$ and denote $Y_i = \cl_G(X_i)$. Note that $Y_i\cap H = X_i$. Then by $H\sqsubseteq G$
\[
\dm(Y_1) + \dm(Y_2) - \dm(Y_1\cap Y_2) \geq \dm(X_1) + \dm(X_2) - \dm(X_1\cap X_2).
\]
By $\dm(Y_i) = \dm(X_i)$ we get $\dm(X_1\cap X_2) \geq \dm(Y_1\cap Y_2)$. The equality is then immediate by $X_1\cap X_2\subseteq Y_1\cap Y_2$.
\end{proof}

\begin{cor}
If $H\sqsubseteq G$ and $G$ is flat, then also $H$ is flat. \qed
\end{cor}

\begin{observation}
$\sqsubseteq$ is transitive.
\end{observation}

\begin{lemma}
\label{lemma: sqsubseteq and sqsubseteq^* interaction}
If $F\sqsubseteq G$ and $F\subseteq H\sqsubseteq^* G$, then $F\sqsubseteq H$.
\end{lemma}

\begin{proof}
Let $\Sigma$ be a finite set of finite dimensional closed subsets of $H$.
Observe that for every $E\in \Sigma$, it holds that $\cl_G(E)\cap F = E\cap F$. Hence, relativizing $\Sigma_G$ to $F$ results in $\Sigma_F$.

By Corollary \ref{corollary: sqsubseteq^* plays nicely with sums}, $\flatsum{G}(\Sigma_G) = \flatsum{H}(\Sigma)$. By $F\sqsubseteq G$, $\flatsum{F}(\Sigma_F)\leq \flatsum{G}(\Sigma_G)$. Therefore, $\flatsum{F}(\Sigma_F) \leq \flatsum{H}(\Sigma)$. We conclude $F\sqsubseteq H$.
\end{proof}


\subsection{Hypergraphs}

A hypergraph $\hg{A}=(M,R)$ is a set of vertices $M$ and a set $R\subseteq [M]^{<\omega}$ of non-empty (hyper)edges. For $P\subseteq M$, we write $R[P]$ for $\setarg{e\in R}{e\subseteq P}$. We write $\hg{B}\subseteq\hg{A}$ if $\hg{B} = (P,R[P])$ for some $P\subseteq M$. We write $\hg{A}[P]$ for the hypergraph $(P,R[P])$ induced on $P$ by $\hg{A}$.

For a finite hypergraph $\hg{A}=(M,R)$ define its \emph{predimension}
\[
\delta(\hg{A}) = |M| - |R|.
\]
For a (possibly infinite) hypergraph $\hg{A}=(M,R)$ define its \emph{associated dimension function} $\dm_{\hg{A}}:\Fin(M)\to\Ints\cup\set{-\infty}$ by
\[
\dm_{\hg{A}}(X) = \inf\setarg{\delta(\hg{B})}{\hg{A}[X] \subseteq \hg{B}\subfin\hg{A}}.
\]
If $\dm_{\hg{A}}$ is non-negative, then it is the dimension function of a pregeometry $G_{\hg{A}}$ on $M$. Call this $G_{\hg{A}}$ the \emph{pregeometry associated to the hypergraph $\hg{A}$}. We say that a set $X$ is closed (independent) in $\hg{A}$ if it is closed (independent) in $G_{\hg{A}}$.
Say that $\hg{A}$ is a \emph{representation} of a pregeometry $G$ if $G_{\hg{A}}=G$. We say that $\hg{A}$ is a \emph{good} representation of $G_{\hg{A}}$ if whenever $e\in R$, then $\dm_{\hg{A}}(e) = |e|-1$.

For $\hg{A} = (M,R)$ and $P\subseteq M$, denote $\delta_{\hg{A}}(P) = \delta(\hg{A}[P])$. For $L\subseteq M$ such that $|L\setminus P| < \infty$ denote $\delta_{\hg{A}}(L/P) = |L\setminus P| - |R[L\cup P]\setminus R[P]|$.
\\
Write $P\strong \hg{A}$ and say that $P$ is \emph{self-sufficient} or \emph{strongly embedded} in $\hg{A}$, if for every finite $X\subseteq M$, $\delta_{\hg{A}}(X/P) \geq 0$. For $\hg{B} = (N,R[N])$ write $\hg{B}\strong \hg{A}$ for $N\strong \hg{A}$. The function $\dm_{\hg{A}}$ being non-negative is equivalent to $\emptyset\strong \hg{A}$. For $P$ finite, $P\strong\hg{A}$ is equivalent to $\dm_{\hg{A}}(P) = \delta(\hg{A}[P])$.

\begin{fact}
\label{fact: self-sufficiency properties}
Let $\hg{A}=(M,R)$. The following are well known and readily follow from the definitions:
\begin{enumerate}
\item
The function $\delta_{\hg{A}}$ is submodular, i.e., $\delta_{\hg{A}}(X/Y) \leq \delta_{\hg{A}}(X/X\cap Y)$ for all $X,Y\subseteq M$.
\item
$\hg{B}\strong \hg{A}$ if and only if $\hg{B}\subseteq \hg{A}$ and $\dm_{\hg{B}}$ is the restriction of $\dm_{\hg{A}}$ to subsets of $\hg{B}$.
\item
The relation $\strong$ is transitive.
\item
If $X$ is closed in $\hg{A}$, then $X\strong \hg{A}$.
\item
If $X,Y\strong \hg{A}$, then $X\cap Y\strong \hg{A}$.
\item
If $X\strong \hg{A}$ and $\delta_{\hg{A}}(Y_i/Y_i\cap X) \leq 0$ for every $i\leq n$, then $X\cup\bigcup_{i=1}^n Y_i\strong \hg{A}$ and the $Y_i$ are freely joined over $X$, i.e., $R[\bigcup_{i=1}^n Y_i\cup X] = R[X]\cup \bigcup_{i=1}^n R[Y_i]$.
\end{enumerate}
\end{fact}

From (5) of Fact \ref{fact: self-sufficiency properties}, we get that every subset $X$ has a \emph{self-sufficient closure} in $\hg{A}$ given by $\bigcap\setcol{Y\supseteq X}{Y\strong \hg{A}}$, which is a non-empty intersection by $\hg{A}\strong \hg{A}$. If $X$ is finite and $\dm_{\hg{A}}$ is bounded from below, then the self-sufficient closure of $X$ is also finite. Denote the self-sufficient closure of the set $X$ in $\hg{A}$ by $\sscl_{\hg{A}}(X)$.

\bigskip
Hrushovski showed that the pregeometry associated to a hypergraph is flat. As we need a slightly stronger statement (the additional part of the proposition below), and as the original proof contains an imprecision\footnote{The last line of the original proof\cite{Hns} implicitly assumes $P_{\emptyset} = \mathbb{P}$ (in the original notation, ${\bigcup_i G_i = G_{\emptyset}}$), which need not be true.} despite being morally correct, we bring the proof in full.
\begin{prop}[{\hspace{-0.001pt}\cite[Lemma 15]{Hns}}]
\label{prop: flatness}
Let $\hg{A}=(M,R)$ be a hypergraph with $\emptyset\strong \hg{A}$ and associated pregeometry $G$. Let $\Sigma =\set{E_1,\dots,E_k}$ be a set of finite dimensional closed sets in $G$. Then
\[
\flatsum{G}(\Sigma) \geq \dm_G(E_{\emptyset}).
\]
Additionally, equality holds if and only if $E_{\emptyset} \strong \hg{A}$ and $R[\bigcup_{i=1}^k E_i] = \bigcup_{i=1}^k R[E_i]$.
\end{prop}

\begin{proof}
Recall that $E_{\emptyset} = \bigcup_{i=1}^k E_i$ need not be closed. For each $s\subseteq [k]$, let $F_s\subseteq E_s$ be a finite set such that $\cl_G(F_s) \supseteq E_s$ and let $\mathbb{F} = \bigcup_{s\subseteq[k]}F_s$. Let $P_{\emptyset}=\sscl_{\hg{A}}(\mathbb{F})$, for each $\emptyset\neq s\subseteq [k]$ let $P_s = P_{\emptyset} \cap E_s$, and let $\mathbb{P} = \bigcup_{i=1}^k P_{\set{i}} = P_{\emptyset}\cap E_{\emptyset}$. Then for every $s\subseteq [k]$
	
\begin{enumerate}
\item
$P_s\strong \hg{A}$, as an intersection of self-sufficient sets;
\item
$F_s\subseteq P_s\subseteq E_s$, hence $\cl_G(P_s) \supseteq E_s$, implying $\delta_\hg{A}(P_s) = \dm_G(E_s)$ by (1).
\end{enumerate}
We compute, using inclusion-exclusion between the third and fourth lines,
\begin{align*}
\dm_G(E_{\emptyset}) - \flatsum{G}(\Sigma) &= \sum_{s\subseteq [k]} (-1)^{|s|}\dm_G(E_s)
\\
&= \sum_{s\subseteq [k]} (-1)^{|s|}\delta_{\hg{A}}(P_s)
\\
&= \sum_{s\subseteq [k]} (-1)^{|s|}|P_s| - \sum_{s\subseteq [k]} (-1)^{|s|}|R[P_s]|
\\
&= (|P_\emptyset| - |\bigcup_{i=1}^k P_i|) - (|R[P_\emptyset]| - |\bigcup_{i=1}^k R[P_{i}]|)
\\
&= |P_\emptyset\setminus \mathbb{P}| - |R[P_\emptyset]\setminus R[\mathbb{P}]| - |R[\mathbb{P}]\setminus\bigcup_{i=1}^k R[P_{i}]|
\\
&= \delta(P_\emptyset/\mathbb{P}) - |R[\mathbb{P}]\setminus\bigcup_{i=1}^k R[P_{i}]|
\end{align*}
Noting that $\dm_{\hg{A}}(\mathbb{P}) = \dm_{\hg{A}}(E_{\emptyset}) = \delta(P_\emptyset)$, the first summand is non-positive, proving the main statement.

We prove the additional part by examining each of the two summands, beginning with the second.

We wish to apply (6) of Fact \ref{fact: self-sufficiency properties} to see that the sets $E_i$ are freely joined over $P_{\emptyset}\strong \hg{A}$. Although each $E_i$ may be infinite, it may be presented as the union of a properly increasing chain of finite sets ${P_i\strong Y_i^1 \strong Y_i^2\strong\dots}$. Since fact \ref{fact: self-sufficiency properties} implies that the sets $Y_i^j$ are freely joined over $P_{\emptyset}$ for every $j$, this is true also for the sets $E_i$ in their entirety. Thus, $R[E_{\emptyset}\cup P_{\emptyset}]=R[P_{\emptyset}]\cup\bigcup_{i=1}^k R[E_i]$. Intersecting both sides with $R[E_{\emptyset}]$, we get $R[E_{\emptyset}]=R[\mathbb{P}]\cup\bigcup_{i=1}^k R[E_i]$. Subtracting $\bigcup_{i=1}^k R[E_i]$ from both resulting sides gives $R[E_{\emptyset}]\setminus \bigcup_{i=1}^k R[E_i]=R[\mathbb{P}]\setminus\bigcup_{i=1}^k R[P_i]$. Thus, the second summand equals zero if and only if $R[E_{\emptyset}] = \bigcup_{i=1}^k R[E_i]$.

The first summand equals zero if and only if $\delta_{\hg{A}}(\mathbb{P}) = \delta_{\hg{A}}(P_{\emptyset}) = \dm_{\hg{A}}(\mathbb{P})$, i.e., if $\mathbb{P}\strong\hg{A}$. We claim that this is equivalent to $E_{\emptyset}\strong\hg{A}$. If $E_\emptyset\strong\hg{A}$, then $\mathbb{P}\strong\hg{A}$ as an intersection of self-sufficient sets. If $\mathbb{P}\strong \hg{A}$, then by using (6) of Fact \ref{fact: self-sufficiency properties} again, this time with $\mathbb{P}$ and the sets $E_i$, we get that $E_\emptyset = \mathbb{P}\cup\bigcup_{i=1}^k E_i\strong \hg{A}$.

As $\dm_G(E_{\emptyset}) = \flatsum{G}(\Sigma)$ if and only if both summands are zero, we are done.
\end{proof}

It is known that every \emph{finite} flat pregeometry has a good representation \cite{EvansMatroid}. The proof uses Hall's Marriage Theorem, which is not applicable to infinite pregeometries, and does not allow the control we will later need. Instead, we give a different proof by inductive construction, that applies to any arbitrary flat pregeometry. We will thus receive the characterization

\begin{theorem}
\label{theorem: flat pregeometry comes from graph}
A pregeometry $G$ has a good representation if and only if $G$ is flat.
\end{theorem}

We only need to show right to left. We execute the construction by laying down one edge at a time. The following is the key lemma allowing us to see the construction through.

\begin{lemma}
\label{lemma: add edge}
Let $G$ be a flat pregeometry on the set $M$. Let $\hg{A} = (M,R)$ be a hypergraph such that $\emptyset\strong \hg{A}$. Assume for some fixed $n$ that
\begin{enumerate}
\item
$R\subseteq [M]^{\leq n+1}$
\item
$\dm_{\hg{A}}(e) = |e|-1$, whenever $e\in R$.
\item
For any $X\subseteq M$, $\cl_{\hg{A}}(X)\subseteq \cl_{G}(X)$. Moreover, if $\dm_G(X)<n$, equality holds.
\end{enumerate}
Assume $S\in [M]^{n}$ is independent in $G$ satisfying 
\begin{enumerate}
\item[$(4)$] If $r\in R\cap [M]^{n+1}$ with $r\nsubseteq \cl_G(S)$, then $\cl_G(r) = \cl_{\hg{A}}(r)$.
\end{enumerate}
such that there exists $t\in \cl_G(S)\setminus\cl_{\hg{A}}(S)$.
Let $\hg{B}=(M,R\cup\set{St})$ be the hypergraph obtained by adding the edge $St:= S\cup \set{t}$ to $\hg{A}$. Then assertions (1)-(3) above hold with respect to $\hg{B}$.
\end{lemma}

\begin{proof}
We begin with assertion $(3)$. It will suffice to show that for any finite $X\subseteq M$, $\cl_{\hg{B}}(X) \subseteq \cl_G(X)$. Assume to the contrary there is some $X\subfin M$ such that there exists $y\in\cl_{\hg{B}}(X)\setminus \cl_G(X)$. Choose $X$ to be of minimal size, hence independent in $\hg{B}$, so also independent in $\hg{A}$.

Let $Y_0=\sscl_{\hg{B}}(Xy)$. It must be that $St\subseteq Y_0$, since otherwise $\delta_{\hg{A}}(Y_0) = \delta_{\hg{B}}(Y_0)$ and $y\in\cl_{\hg{A}}(X)\subseteq \cl_G(X)$. In particular, $St\subseteq Y_0$ implies $Y_0\strong \hg{A}$. Define
\[
\Sigma = \setarg{\cl_G(r)}{r\in R[Y_0]\cup\set{St}}
\]
and, on account of Observation \ref{observation: inclusion in flatsum}, thin $\Sigma$ out by removing any $Z\in\Sigma$ that is not maximal under inclusion in $\Sigma$. Observe that $\cl_G(St)$ remains in $\Sigma$, because it is of dimension $n$, which is maximal in $\Sigma$ by assertion $(1)$. Enumerate $\Sigma = \set{E_1,\dots,E_{k+1}}$ with $E_{k+1} = \cl_G(St)$. For each $1\leq i\leq k$ fix some $r_i\in R[Y_0]$ such that $E_i = \cl_G(r_i)$, and observe that by assertion $(4)$, also $E_i=\cl_{\hg{A}}(r_i)$.

We would like to have for each non-empty $s\subseteq [k+1]$ the equality
\[
\cl_G(E_s\cap Y_0) = E_s
\]
To achieve this, we will replace $Y_0$ with a bigger set $Y$ such that $\delta_{\hg{B}}(Y) = \dm_{\hg{B}}(Xy)$, $R[Y] \subseteq \bigcup_{i=1}^{k+1} R[E_i]$, and the equality above holds for every $s\subseteq[k+1]$.

The equality already holds whenever $|s| = 1$, and increasing $Y_0$ will not change that. For each $s\subseteq [k+1]$ with $|s|\geq 2$ fix some $m_s\in s$ such that $m_s\neq k+1$, and choose arbitrarily $B_s\subfin E_s$ such that $E_s\subseteq \cl_G(B_s)$.
As $B_s\subseteq \cl_G(r_{m_s}) = \cl_{\hg{A}}(r_{m_s})$, let $D_s\subseteq E_{m_s}$ be such that $B_s\subseteq D_s$, $E_{m_s}\cap Y_0\subseteq D_s$, and $\delta_{\hg{A}}{(D_s/D_s\cap Y_0)} \leq 0$. Denote
\[
Y=Y_0\cup\bigcup\setarg{D_s}{s\subseteq [k+1], |s|\geq 2},
\]
and note we have $\delta_{\hg{A}}(Y/Y_0) \leq 0$. Moreover, $Y_0\strong \hg{A}$ implies that equality holds, $Y\strong \hg{A}$, and for each $r\in R(Y)\setminus R(Y_0)$ there is some $s$ with $r\in R[D_s]$ (using (6) of \ref{fact: self-sufficiency properties}). In particular, for each such an $r$ there is some $i\leq k$ such that $r\subseteq E_i\in \Sigma$.

Denote $F_i = E_i\cap Y$ and let $\Sigma_{Y} = \set{F_1,\dots,F_{k+1}}$. 
We observe that every element of $Y$ appears in some $r\in R[Y]\cup\set{St}$. For elements of $Xy$, this is by assumption on $y$ and minimality of $X$. For $a\in Y\setminus Xy$, since $\delta_{\hg{A}}(Y) = \dm_G(Xy)\leq \dm_{\hg{A}}(Y)$ it must be that $\delta_{\hg{A}}(Y\setminus \set{a}) \geq \delta_{\hg{A}}(Y)$ so $a\in r$ for some $r\in R[Y]\cup\set{St}$. Thus, $F_\emptyset = Y$.

By flatness of $G$,
\[
\dm_G(Y) = \dm_G(E_\emptyset) \leq \flatsum{G}(\Sigma).
\]
By $\dm_G(E_s) = \dm_G(F_s)$ for every $\emptyset\neq s \subseteq[k+1]$,
\[
\flatsum{G}(\Sigma) = \sum_{\emptyset\neq s\subseteq [k+1]} (-1)^{|s|+1}\dm_G(F_s).
\]

Denote $\hg{Y} = \hg{A}[Y]$. Since $Y\strong \hg{A}$, we have $\dm_{\hg{Y}} = \dm_{\hg{A}}\restrictedto \Powerset(Y)$, as well as $\cl_{\hg{Y}}(Z) = \cl_{\hg{A}}(Z)\cap Y$ for any $Z\subseteq Y$. For each $i\leq k$, recall that $E_i = \cl_{\hg{A}}(r_i)$. For each $s\subseteq [k+1]$ such that $|s|\geq 2$, because $\dm_G(E_s)<n$, by assertion $(3)$ the set $E_s$ is closed in $\hg{A}$ and $\dm_{\hg{A}}(E_s) = \dm_G(E_s)$. Therefore, in either case, $F_s$ is closed in $\hg{Y}$ with $\dm_{\hg{Y}}(F_s) = \dm_{\hg{A}}(F_s) = \dm_G(F_s)$. Then
\begin{align*}
\flatsum{G}(\Sigma) &= \left(\sum_{\substack{\emptyset\neq s\subseteq [k+1]\\s\neq \set{k+1}}} (-1)^{|s|+1}\dm_{\hg{Y}}(F_s)\right) + \dm_G(F_{k+1})
\\
&\leq \left(\sum_{\substack{\emptyset\neq s\subseteq [k+1]\\s\neq \set{k+1}}} (-1)^{|s|+1}\dm_{\hg{Y}}(F_s)\right) + (\dm_{\hg{A}}(F_{k+1})-1)
\\
&= \left(\sum_{\emptyset\neq s\subseteq [k+1]} (-1)^{|s|+1}\dm_{\hg{Y}}(F_s)\right) - 1
\\
&= \flatsum{\hg{Y}}(\Sigma_Y) - 1
\end{align*}

For any edge $r\in R$ there is some $i\leq k+1$ such that $r\in F_i$, and $F_\emptyset\strong \hg{Y}$ because $F_\emptyset = Y$. These two facts, by the additional part of Proposition \ref{prop: flatness}, imply that $\flatsum{\hg{Y}}(\Sigma_Y) = \dm_{\hg{Y}}(F_\emptyset)$. Hence,
\begin{align*}
\flatsum{G}(\Sigma) &\leq \dm_{\hg{Y}}(F_\emptyset) - 1
\\
&= \dm_{\hg{A}}(Y) - 1
\\
&= (\dm_{\hg{B}}(Y) + 1) -1 = \dm_{\hg{B}}(Y)
\end{align*}
We conclude that $\dm_G(Y)\leq \dm_{\hg{B}}(Y) = |X|$, so $y\in \cl_G(X)$ in contradiction to our assumption.

This proves assertion $(3)$ holds in $\hg{B}$. Assertion $(1)$ is immediate by construction. Lastly, assertion $(3)$ gives us that $\dm_{\hg{B}}(e) = \dm_G(e) = \dm_{\hg{A}}(e) = |e|-1$ for every $e\in R$. As $\dm_{\hg{B}}(St) = n$ by construction, assertion $(2)$ holds in $\hg{B}$ as well.
\end{proof}

Now, to prove Theorem \ref{theorem: flat pregeometry comes from graph} all that we need is an enumeration of all finite tuples in $G$ that is favorable to applications of Lemma \ref{lemma: add edge}.

\begin{definition}
Let $G$ be a flat pregeometry on a set $M$, let $X$ be the set of all pairs $(A, F)$ where $F$ is a closed set in $G$ and $A\in[F]^{\dm_G(F)+1}$. Let $(I,<)$ be a well ordering Consider a bijection $f: I\to X$, and denote by $f_1$, $f_2$ its projections, i.e., for each $i\in I$, $f(i) = (f_1(i), f_2(i))$.
For every $n\in\Nats$, denote $I^f_n:=\setarg{i\in I}{f_1(i)\in [M]^n}$ and
for every $F$, a closed set in $G$, denote $I^f_F:=\setarg{i\in I}{f_2(i)=F}$.

Say that $f$ is a \emph{valid enumeration} for $G$ if the sets of the form $I^f_n, I^f_F$ are intervals in $I$ such that $I^f_n < I^f_m$ whenever $n<m$.
\end{definition}

\begin{definition}[The Enumerative Construction]
\label{definition: enumerative construction}
Let $G$ be a flat pregeometry on a set $M$ and let $f$ be a valid enumeration for $G$. We define an inductive construction of a hypergraph, one edge at a time. Let $\hg{A}_0 = (M,\emptyset)$. For each $i\in I$, we construct $\hg{A}_i = (M, R_i)$ such that, denoting $n= |f_1(i)| - 1$,
\begin{enumerate}[i.]
	\item
	If $i<j$, then $R_i\subseteq R_{j}$
	\item
	Assertions $(1)-(3)$ of lemma \ref{lemma: add edge} hold with respect to $n$, $\hg{A}_i$ and $G$.
	\item
	$\dm_{\hg{A}_{i+1}}(f_1(i)) = \dm_G(f_1(i))$.
\end{enumerate}
At stage $i$, check whether $\dm_{\hg{A}_i}(f_1(i)) = \dm_G(f_1(i))$. If equality holds, we define $R_{i+1} = R_{i}$. Otherwise, by assertion $(3)$ it must be that $f_1(i)$ is independent in $\hg{A}_i$. Let $S=f_1(i)\setminus\set{t}$ for some $a\in f_1(i)$,  then $t\in \cl_G(S)\setminus \cl_{\hg{A}_i}(S)$. Additionally, for every $r\in R_i$, letting $F=\cl_G(r)$, since all of $I^f_F$ has already been enumerated, $\cl_{\hg{A}_i}(r) = \cl_G(r)$, i.e., (4) of Lemma \ref{lemma: add edge} holds with respect to $S$. Therefore we define $R_{i+1} = R_i\cup\set{f_1(i)}$. For $i\in I$ a limit in $(I, <)$, define $R_{i} = \bigcup \setcol{R_{j}}{j< i}$.

It is easy to verify that the conditions above hold, with the second condition given by Lemma \ref{lemma: add edge}. Letting $R = \bigcup_{i\in I} R_i$, we define $\hg{A}_f = (M, R)$, \emph{the $f$-construction of $G$}.
\end{definition}

The next proposition concludes the proof of Theorem \ref{theorem: flat pregeometry comes from graph}:

\begin{prop}
For $G$ a flat pregeometry on a set $M$ and $f$ a valid enumeration for $G$, the hypergraph $\hg{A}_f$ is a good representation of $G$.
\end{prop}

\begin{proof}
Let $X\subfin M$. Since $f$ is surjective, let $i\in I$ be such that $f_1(i) = X$. Then for every $j>i$, Lemma \ref{lemma: add edge} guarantees $\dm_{\hg{A}_j}(X) = \dm_G(X)$, so in particular $\dm_{\hg{A}_f}(X) = \dm_G(X)$. Conclude that $G_{\hg{A}_f} = G$.
\end{proof}

\begin{observation}
\label{observation: edge sets in enumerative construction}
\begin{enumerate}[i.]
\item
Every $\hg{B} = (M, S)$, a good representation of $G$, can be attained as an enumerative construction --- choose $f$ such that for every closed set $F$, the set $\setcol{(A,F)}{A\in S\cap [F]^{\dm_G(F)+1}}$ is an initial segment of $I^f_F$.
\item
Let $\hg{A}_1 = (M, R_1)$ and $\hg{A}_2 = (M,R_2)$ be two good representations of $G$ and let $F$ be a closed $n$-dimensional subset of $G$. Denoting $R_i^F = R_i\cap [F]^{n +1}$ for $i\in\set{1,2}$, the hypergraph $(M, (R_1\setminus R_1^F)\cup R_2^F)$ is also a good representation of $G$. See this by replacing the enumeration of $[F]^{n+1}$ used for the construction of $\hg{A}_1$ with that used for the construction of $\hg{A}_2$.
\end{enumerate}
\end{observation}

Our next goal is to attach hypergraph characterizations to the geometric embeddings $\sqsubseteq^*$ and $\sqsubseteq$. We do this by carefully choosing the enumeration used in the construction.

\begin{definition}
\label{definition: hydra and centered at}
Let $G$ be a flat pregeometry on a set $M$ and let $f$ be a valid enumeration for $G$ with domain $I$.
\begin{enumerate}
\item
Say that $f$ is a \emph{hydra} if for each $F$ closed in $G$ there exists an independent set $Z^f_F\in [F]^{\dm_G(F)}$ such that the set $\setarg{i\in I}{f_1(i) = Z^f_F\cup\set{a},\ a\in F}$ is an initial segment of $I^f_F$.
\item
For a set $P\subseteq M$, say that $f$ is \emph{centered} at $P$ if for every $F$ closed in $G$, the sets $\setarg{i\in I^f_F}{f_1(i)\subseteq P}$, $\setarg{i\in I^f_F}{\dm_G(f_1(i)\cap P) = \dm_G(F\cap P)}$ are initial segments of $I^f_F$. If $f$ is a hydra, require also that $|Z^f_F\cap P| = \dm_G(F\cap P)$.
\item
For $P\subseteq M$, letting $J=\setarg{i\in I}{f_1(i)\subseteq P}$, call the function $f|_J$ the \emph{restriction of $f$ to $P$}.
\end{enumerate}
\end{definition}

\begin{observation}
Let $G$ be a flat pregeometry defined on a set $M$ and let $H\subseteq G$ be the restriction of $G$ to $P\subseteq M$. Assume $H$ is flat. Then whenever $f$ is a valid enumeration for $G$, the restriction of $f$ to $P$ is a valid enumeration for $H$.
Moreover, if $f$ is a hydra for $G$ centered at $P$, then the restriction of $f$ to $P$ is a hydra for $H$.
\end{observation}

Using a hydra $f$ for the enumeration makes the resulting hypergraph easier to understand. For a closed set $F$, looking only at the edges $e\subseteq F$ with $|e| = \dm(F)+1$, we get a (partial) ``sun'' shape, with the edges connecting elements of $F$ to the basis $Z^f_F$. Those elements who do not have an edge going to them, are already in the closure of the ``sun'' due to existing edges of lower dimension.

\begin{lemma}
\label{lemma: edges contain hydra root}
Let $G$ be a flat pregeometry on $M$, let $f$ be a hydra for $G$ with domain $I$. Then for $F$ closed in $G$, if $e$ is an edge in $A_f$ such that $\cl_G(e) = F$, then $Z^f_{F}\subseteq e$.

In particular, if $f$ is centered at $P\subseteq M$, $\dm_G(e\cap P) = \dm_G(F\cap P)$.
\end{lemma}

\begin{proof}
Let $F$ be an $n$-dimensional closed set in $G$ and let $i\in I^f_F$. It will be enough to show that, if $Z^f_F\nsubseteq f_1(i)$, then when $f_1(i)$ is examined during the construction of $\hg{A}_f$, already $\dm_{\hg{A}_i}(f_1(i)) \leq n$.

Let $a\in f_1(i)\setminus Z^f_F$. Since we have examined the set $Z^f_F\cup\set{a}$ in a previous stage, we have $a\in\cl_{\hg{A}_i}(Z^f_F)$. Thus, $f_1(i)\subseteq \cl_{\hg{A}_i}(Z^f_F)$ and $\dm_{\hg{A}_i}(f_1(i)) \leq \dm_{\hg{A}_i}(Z^f_F) = n$.
\end{proof}

This subsection culminates in the proof of Corollary \ref{corollary: sqsubseteq is equivalent to self-sufficient}, summarizing the analogy between $\sqsubseteq$ and $\strong$. Propositions \ref{prop: self-sufficiency implies sqsubseteq} and \ref{prop: sqsubseteq implies leqslant} are the two directions of the proof.

\begin{prop}
\label{prop: self-sufficiency implies sqsubseteq}
Let $\hg{A} = (M,R)$, $\hg{B}=(P,R[P])$ be such that $\hg{B}\strong \hg{A}$. Then $G_{\hg{B}} \sqsubseteq G_{\hg{A}}$. In particular, $G_{\hg{B}} \sqsubseteq^* G_{\hg{A}}$.
\end{prop}

\begin{proof}
Let $\Sigma$ be a finite collection of finite dimensional closed subsets of $G_{\hg{A}}$.
Define $R' = \bigcup_{E\in \Sigma} R[E]$ and consider the hypergraphs $\hg{A}' = (M,R')$, $\hg{B}' = (P,R'[P])$ obtained by removing from $\hg{A}$ and $\hg{B}$ all edges not contained in some $E\in\Sigma$. Observe the following easy facts:
\begin{enumerate}
\item
Since $R'\subseteq R$, if $Y\strong \hg{A}$, then also $Y\strong \hg{A}'$.
\item
For any $E\in\Sigma$ and $Y\strong \hg{A}[E]$, because $R'[Y] = R[Y]$, we have that $\dm_{\hg{A}}(Y) = \dm_{\hg{A}'}(Y)$.
\item
Because $R' = R'[\bigcup\Sigma]$, clearly $\bigcup \Sigma \strong \hg{A}'$.
\end{enumerate}

From applying $(1)$ to $P$, we get $\hg{B}'\strong \hg{A}'$, hence $G_{\hg{B}'} \subseteq G_{\hg{A}'}$ by (2) of Fact \ref{fact: self-sufficiency properties}. Applying $(2)$ to $E_s:= \bigcap S$, for any non-empty $S\subseteq \Sigma$, we get $\dm_{\hg{A}}(E_s) = \dm_{\hg{A}'}(E_s)$ and $\dm_{\hg{B}}(E_s\cap P) = \dm_{\hg{A}}(E_s\cap P) = \dm_{\hg{A}'}(E_s\cap P) = {\dm_{\hg{B}'}(E_s\cap P)}$, where $(2)$ is again used for the middle equality. Therefore, we see that $\flatsum{\hg{A}}(\Sigma) = \flatsum{\hg{A}'}(\Sigma)$ and ${\flatsum{\hg{B}}(\Sigma_P) = \flatsum{\hg{B}'}(\Sigma_P)}$. By the additional part of Proposition \ref{prop: flatness}, we have ${\flatsum{\hg{A}'}(\Sigma) = \dm_{\hg{A}'}(\bigcup \Sigma)}$ and ${\flatsum{\hg{B}'}(\Sigma_P) = \dm_{\hg{B}'}(\bigcup\Sigma_P) = \dm_{\hg{A}'}(\bigcup \Sigma_P)}$, where the last equality is by $(3)$. By monotonicity, $\dm_{\hg{A}'}(\bigcup \Sigma_P) \leq \dm_{\hg{A}'}(\bigcup\Sigma)$, so $\flatsum{\hg{B}}(\Sigma_P) \leq \flatsum{\hg{A}}(\Sigma)$. We conclude that $G_{\hg{B}} \sqsubseteq G_{\hg{A}}$.
\end{proof}


\begin{lemma}
\label{lemma: flatsum covering edges}
Let $\hg{A} = (M,R)$. Let $X\subfin M$, $Y\supseteq X$ be such that $\delta_{\hg{A}}(Y) = \dm_{\hg{A}}(X)$, and let $\Gamma = \setarg{\cl_{\hg{A}}(r)}{r\in R[Y]}$. If $\Sigma\supseteq \Gamma$ is a finite collection of finite dimensional closed subsets of $\hg{A}$ with $X\subseteq \bigcup \Sigma$ and $\dm_{\hg{A}}(E\cap Y) = \dm_{\hg{A}}(E)$ for every $E\in \Sigma$, then $\flatsum{\hg{A}}(\Sigma) = \dm_{\hg{A}}(X)$.

In particular, if $Y=M$ and $\Sigma\supseteq \Gamma$, then $\flatsum{\hg{A}}(\Sigma) = \dm_{\hg{A}}(M)= \delta(\hg{A})$.
\end{lemma}

\begin{proof}
Denote $\hg{B} = \hg{A}[Y]$. By Proposition \ref{prop: self-sufficiency implies sqsubseteq}, $G_{\hg{B}}\sqsubseteq^* G_{\hg{A}}$. Since by assumption $\Sigma = \setarg{\cl_{\hg{A}}(F)}{F\in \Sigma_Y}$, Corollary \ref{corollary: sqsubseteq^* plays nicely with sums} gives $\flatsum{\hg{A}}(\Sigma) = \flatsum{\hg{B}}(\Sigma_Y)$.

Every $y\in Y\setminus X$ is involved in some relation $r\in R[Y]$, or otherwise we would have $\delta(\hg{B}) > \dm_{\hg{A}}(X)$. Thus, $\bigcup \Sigma_Y = Y\strong \hg{B}$, and by definition of $\Gamma$, $R[Y]\subseteq \bigcup_{E\in \Sigma}R[E]$. Then Proposition \ref{prop: flatness} guarantees that $\flatsum{\hg{B}}(\Sigma_Y) = \dm_{\hg{B}}(Y) = \dm_{\hg{A}}(X)$.
\end{proof}


\begin{prop}
\label{prop: sqsubseteq implies leqslant}
Let $H\subseteq G$ be flat pregeometries on $P\subseteq M$, respectively. Let $f$ be a hydra for $G$ centered at $P$ with $\Domain(f)=I$. Let $g$ be the restriction of $f$ to $P$, and denote $J=\Domain (g)$. Let $\hg{A} = (M, R)$ be the $f$-construction of $G$ and let $\hg{B} =(P, S)$ be the $g$-construction of $H$. Then
\begin{enumerate}
\item 
If $H\sqsubseteq^* G$, then $R[P] \subseteq S$.
\item
$H\sqsubseteq G$ if and only if $R[P] = S$ if and only if $\hg{B} \strong \hg{A}$.
\end{enumerate}
\end{prop}

\begin{proof}
Assume first that $H\sqsubseteq^* G$. For each $j\in J\subseteq I$, denote by $\hg{A}_j = (M, R_j)$ and $\hg{B}_j = (P, S_j)$ the $j$-th stages of the construction of $\hg{A}$ and $\hg{B}$, respectively.

\medskip
\noindent\textbf{Claim.} For each stage $j\in J$ of the construction, whenever $Y\subseteq P$ is finite, then $\dm_{\hg{A}_j}(Y) \leq \dm_{\hg{B}_j}(Y)$.

\begin{proof}
We prove by induction. This is clear for $j:=\min J$. If $j\in J$ is a limit stage and $\dm_{\hg{A}_j}(Y) > \dm_{\hg{B}_j}(Y)$ for some $Y\subseteq P$, then there is already some successor stage $j'<j$ in which $R_{j'}[\sscl_{\hg{A}_j}(Y)] = R_{j}[\sscl_{\hg{A}_j}(Y)]$ and $S_{j'}[\sscl_{\hg{B}_j}(Y)] = S_{j}[\sscl_{\hg{B}_j}(Y)]$, implying $\dm_{\hg{A}_{j'}}(Y) > \dm_{\hg{B}_{j'}}(Y)$. Thus, we only need to take care of successor stages.

Assume the claim is true for stage $j\in J$, but not $j+1$. Let $Y\subseteq P$ be finite such that $\dm_{\hg{A}_{j+1}}(Y) > \dm_{\hg{B}_{j+1}}(Y)$. By replacing it with its self-sufficient closure in $\hg{B}_{j+1}$, we may assume that $Y\strong \hg{B}_{j+1}$, and so $Y\strong \hg{B}_{j}$. Let $e := f_1(j) = g_1(j)$, $n := |e|-1$, and $F=\cl_G(e)$. It must be that $e$ is an edge in $\hg{B}_{j+1}$, but not in $\hg{A}_{j+1}$, and $e\subseteq Y$.

Let $\Gamma_j$ be the collection of closed sets in $H$ whose tuples we finished enumerating prior to stage $j$ of the construction. Let $\Sigma_H = \setarg{E\in \Gamma_j}{\dm_H(E\cap Y) = \dm_H(E)}$. For each $E\in \Sigma_H$, denote $E_G =\cl_G(E)$. Then for each $E\in \Sigma_H$, $\cl_{\hg{B}_j}(E)=E$, $\cl_{\hg{A}_j}(E) = E_G$, and $\dm_{\hg{B}_j}(E) = \dm_{\hg{A}_j}(E_G) = \dm_G(E_G)$. By $H\sqsubseteq^* G$, we have $\dm_{H}(E\cap F) = \dm_{G}(E_G\cap F)$. Again by $H\sqsubseteq^* G$ and choice of $\Gamma_j$, using Corollary \ref{corollary: sqsubseteq^* plays nicely with sums}, we have $\flatsum{\hg{B}_j}(\Sigma_H) = \flatsum{H}(\Sigma_H) = \flatsum{G}(\Sigma_G) = \flatsum{\hg{A}_j}(\Sigma_G)$.

We'd like to have $\dm_{\hg{B}_j}(E\cap F\cap Y) = \dm_{\hg{A}_j}(E_G\cap F)$, for every $E\in \Sigma_H$. We achieve this by increasing $Y$ to a superset $\bar{Y}$. For each $E\in \Sigma_H$, let $K_E\subseteq E$ be finite such that $\cl_G(K_E) \supseteq E\cap F$. In particular, $\cl_G(K_E) = E_G\cap F$. Let $L_E \subseteq E$ be finite containing $K_E$ such that $\delta_{\hg{B}_j}(L_E/E\cap Y)\leq 0$. Let $\bar{Y}$ be $Y\cup\bigcup_{E\in\Sigma} L_E$. Since $Y\strong \hg{B}_j$ and $\delta_{\hg{B}_j}(\bar{Y}/Y)\leq 0$, we have $\bar{Y}\strong \hg{B}_j$ and every edge in $S_j[\bar{Y}]\setminus S_j[Y]$ is in $S_j[L_E\cup Y]$ for some $E\in\Sigma_H$. Observe that still $\dm_{\hg{A}_{j+1}}(\bar{Y}) \geq \dm_{\hg{A}_{j+1}}(Y) > \dm_{\hg{B}_{j+1}}(Y) = \dm_{\hg{B}_{j+1}}(\bar{Y})$.

Let $K$, containing $e$, be a basis for $\bigcup_{E\in \Sigma_H} K_E$ in $\hg{B}_j$. Observe that since $e$ is not an edge in $\hg{A}_{j+1}$, the set $K$ is not independent in $\hg{A}_j$, hence $\dm_{\hg{A}_j}(K) < \dm_{\hg{B}_j}(K)$. Let $C = \cl_{\hg{A}_j}(K)$, $D = \cl_{\hg{B}_j}(K)$, and note that $C = \cl_{\hg{A}_j}(D)$. For each $E\in \Sigma_H$, as $K_E\subseteq C, D$, we have $\dm_{\hg{B}_j}(D\cap E) = \dm_{\hg{A}_j}(C\cap E_G) = \dm_{G}(F\cap E_G)$. Since $e\subseteq D$ and $g$ is a hydra, $Z^f_F\subseteq e \subseteq D$, and so Lemma \ref{lemma: edges contain hydra root} gives us that any $e'\in S_j[F]$ with $|e'| = n+1$, is contained in $D$.
Now all conditions of Lemma \ref{lemma: flatsum covering edges} hold with respect to $\hg{B}_j$, $Y$, $\bar{Y}$, and $\Sigma\cup \set{D}$. Thus, $\flatsum{\hg{B}_j}(\Sigma_H\cup \set{D}) = \dm_{\hg{B}_j}(Y) = \dm_{\hg{B}_j}(\bar{Y})$.
Note that since $Y\strong \hg{B}_{j}$, the set $\Sigma_H$ is closed under intersections.
Thus,
\begin{align*}
\flatsum{\hg{A}_j}(\Sigma_G \cup \set{C}) - \flatsum{\hg{B}_j}(\Sigma_{H} \cup \set{D}) &= \dm_{\hg{A}_j}(C) - \dm_{\hg{B}_j}(D) < 0.
\end{align*}
On the other hand, flatness of the pregeometry associated to $\hg{A}_j$ gives $\dm_{\hg{A}_j}(\bar{Y}) \leq \flatsum{\hg{A}_j}(\Sigma_G \cup \set{C})$. We conclude that $\dm_{\hg{A}_j}(\bar{Y}) < \dm_{\hg{B}_j}(\bar{Y})$. Since $\dm_{\hg{A}_{j+1}}(\bar{Y}) = \dm_{\hg{A}_j}(\bar{Y})$ and $\dm_{\hg{B}_{j+1}}(\bar{Y}) = \dm_{\hg{B}_j}(\bar{Y}) - 1$ this is in contradiction to $\dm_{\hg{A}_{j+1}}(\bar{Y}) > \dm_{\hg{B}_{j+1}}(\bar{Y})$.
\end{proof}

Part (1) of the main statement is now clear. For every $j\in J$, if $e:= f_1(j) = g_1(j)$ is an edge in $R[P]$, then $\dm_{\hg{B}_j}(e) \geq \dm_{\hg{A}_j}(e) = |e|$, so by construction $e\in S_{j+1}\subseteq S$.

We prove part (2). We assume $S\setminus R[P]\neq\emptyset$ and show $H\not\sqsubseteq G$. We may assume $H\sqsubseteq^* G$, for otherwise this is clear. Let $e\in S\setminus R[P]$ be such that $j:=\min f_1^{-1}(e)$ is minimal, i.e., $S_j=R_j[P]$. Denote $n:= |e|-1$, let $X=\sscl_{\hg{A}_j}(e)$, and observe that $X\cap P\strong \hg{B}_j$. If $X\subseteq P$, then $\delta_{\hg{B}_j}(X\cap P) = \delta_{\hg{A}_j}(X) = \dm_{\hg{A}_j}(X) = n$, which is not the case, because $e\in S_{j+1}$.

We strip away from $\hg{A}_j$ and $\hg{B}_j$ the edges of full dimension in $\cl_G(e)$. Denote $R' = \setarg{r\in R_j}{\cl_{\hg{A}_j}(r)\neq \cl_{G}(e)}$, $S' = R'[P]$ and let $\hg{A}' = (M, R')$, $\hg{B}' = (P,S')$. Note that by Lemma \ref{lemma: edges contain hydra root} and $f$ being centered at $P$, it must be that ${R_j[X]\setminus R' = S_j\setminus S'\subseteq [P]^{n+1}}$. Denoting $m=|S_j[X]\setminus S'|$, the number of edges we have removed, we have $\dm_{\hg{A}'}(X) = \delta_{\hg{A}'}(X) = n+m$ and $\dm_{\hg{B}'}(X\cap P) = \delta_{\hg{B}'}(X\cap P) = n+1+m$.

By construction, for every $k<n$, $G$ and $\hg{A}'$ have exactly the same $k$-dimensional subsets. The same goes for $H$ and $\hg{B}'$. Let $\Sigma_G = \setarg{\cl_{G}(X_0)}{X_0\in[X]^{<n}}$. Then $\flatsum{G}(\Sigma_G) = \flatsum{\hg{A}'}(\Sigma_G)$ as well as $\flatsum{H}(\Sigma_H) = \flatsum{\hg{B}'}(\Sigma_H)$. Since $X\strong \hg{A}'$ and $R'[X]\subseteq [X]^{\leq n}$, by Lemma \ref{lemma: flatsum covering edges} (for the special case $X=Y$, in the notation of the lemma), we have $\dm_{\hg{A}'}(X) = \flatsum{\hg{A}'}(\Sigma_G)$. Similarly, $\dm_{\hg{B}'}(X\cap P) = \flatsum{\hg{B}'}(\Sigma_H)$. Conclude that $\flatsum{G}(\Sigma_G) < \flatsum{H}(\Sigma_H)$ and $H\not\sqsubseteq G$.

If $R[P]=S$, then $\hg{B}\subseteq \hg{A}$. The restriction of $\dm_{\hg{A}}= \dm_{G}$ to subsets of $P$ is precisely $\dm_H = \dm_{\hg{B}}$. By (2) of Fact \ref{fact: self-sufficiency properties}, this means $\hg{B}\strong \hg{A}$.

Finally, $\hg{B}\strong \hg{A}$ implies $H\sqsubseteq G$ by Proposition \ref{prop: self-sufficiency implies sqsubseteq}.
\end{proof}

\begin{cor}
\label{corollary: sqsubseteq is equivalent to self-sufficient}
Let $G$ be a flat pregeometry and let $H\subseteq G$. Then $H\sqsubseteq G$ if and only if there exist hypergraphs $\hg{B}\strong \hg{A}$ such that $G_{\hg{B}} = H$ and $G_{\hg{A}} = G$. \qed
\end{cor}

%

\subsection{The $\alpha$-function}

The $\alpha$-function was defined by Mason \cite{Onaclassofmatroidsarisingfrompathsingraphs} in order to characterize the class of flat pregeometries  --- strict gammoids, in matroid theoretic terminology. However, Mason's definition of a strict gammoid was distinct from ours, going through linkages in directed graphs. From the point of view of our presentation, the $\alpha$ function is a measure of how much the dimension of a set deviates from the sum of ``dimensional data'' contained in its subsets of smaller dimension, under the assumption ``flat'' interaction between these lower dimensional subsets.

\begin{definition}
In the context of an ambient pregeometry $G$, for $X\subseteq G$, write $Y\subflat X$ to indicate that $Y$ is a closed set in $G$ such that $Y\subseteq X$. Write $Y\psubflat X$ to mean $Y\subflat X$ and $Y\neq X$.
\end{definition}

\begin{remark}
When there is ambiguity with respect to the ambient pregeometry with which $Y\subflat X$ is used, we dispel it like so: $Y\subflat X\subseteq H$.
\end{remark}

\begin{definition}
Let $G$ be a pregeometry. For every $X\subseteq G$ finite, define recursively
\[
\alpha_{G}(X) = |X|-\dm(X) - \sum_{Y\psubflat X}\alpha_G(Y)
\]
In this definition an empty sum is taken to equal zero.
\end{definition}

In the flat context, $\alpha(X)$ is the number of ``edges'' that must be put on $X$, on top of edges contained in its closed proper subsets, in order to achieve its dimension.

\begin{prop}
\label{prop: alpha counts edges}
Let $G$ be a flat pregeometry on $M$. Let $\hg{A}=(M,R)$ be a representation of $G$ and let $F$ be a finite closed subset of $G$. Then $\alpha_G(F) = |\setarg{r\in R}{\cl(r) = F}|$.
\end{prop}

\begin{proof}
By induction on $\dm(F)$.

\underline{$\dm(F) = 0$}: $F=\cl(\emptyset)$ so $F\strong \hg{A}$ and $0 = \dm(F) = \delta(F) = |F|-|R[F]|$, hence $|R[F]| = |F|$. On the other hand, $\alpha_G(F) = |F| -\dm(F) - \sum_{X\psubflat F} \alpha(X) = |F|$.

\underline{$\dm(F) > 0$}: For any $r\in R[F]$ such that $\cl(r)\neq F$, the set $\cl(r)$ is a closed proper subset of $F$. Therefore, by induction hypothesis, the number of edges on $F$ whose closure is not $F$ is precisely $\sum_{X\psubflat F} \alpha_G(X)$. Also, since $F\strong \hg{A}$, we have $\dm(F) = \delta(F) = |F| - |R[F]|$, so $|R[F]| = |F|-\dm(F)$. We get that the number of edges in $R[F]$ whose closure is $F$ is
\[
|R[F]| - \sum_{X\psubflat F} \alpha(X) = |F| - \dm(F) - \sum_{X\psubflat F} \alpha_G(X) = \alpha_G(F)
\]
\end{proof}

The weakness of the $\alpha$-function is that it only sees finite closed sets. In his preprint, Evans \cite{EvansMatroid} explores the connection between flatness, hypergraphs and the $\alpha$-function for finite pregeometries. The following characterization of flatness can be found in section 4.

\begin{prop}
\label{prop: characterization of flatness via alpha}
A finite pregeometry $G$ is flat if and only if whenever $X$ is a union of closed sets, then $\alpha_G(X)\geq 0$.
\end{prop}

To better understand the $\flatsum{}$ operation, and for the sake of completeness, we strengthen the key lemma \cite[Lemma 4.2]{EvansMatroid} and bring the proof of Proposition \ref{prop: characterization of flatness via alpha} in full.

In light of Proposition \ref{prop: alpha counts edges}, Lemma \ref{lemma: true meaning of alternating sum} is best understood in the setting of $G$ flat, and holding in mind some good representation of $G$. In that case, the alternating sum $\flatsum{G}(\Sigma)$ is truly an inclusion-exclusion on sets of edges. With that said, the lemma holds also when $G$ is not flat.

\begin{lemma}
\label{lemma: true meaning of alternating sum}
Let $G$ be a finite pregeometry. Let $\Sigma$ be a collection of closed subsets of $G$ such that if $Y\subflat X\in \Sigma$, then $Y\in \Sigma$. Then
\[
\flatsum{G}(\Sigma) - \dm(G) = \sum_{\substack{X\subflat G\\X\notin \Sigma}} \alpha_G(X) - |G\setminus \bigcup \Sigma|
\]
\end{lemma}

\begin{proof}
We construct $\Sigma$ inductively and show that the equation holds with respect to each intermediate stage. Denote $F=\cl(\emptyset)$ and let $\Sigma_0 = \set{F}$. Observe $\flatsum{G}(\Sigma_0) = 0$ and recall $\alpha_G(F) = |F|$. Then
\[
\sum_{\substack{X\subflat G\\X\neq F}} \alpha_G(X) - |G\setminus F| = \sum_{X\subflat G} \alpha_G(X) - |G| = -\dm(G) = \flatsum{G}(\Sigma_0)-\dm(G)
\]
where the second equality is by definition of $\alpha_G(G)$.

Assume now that we have constructed $\Sigma_i\subseteq \Sigma$, downwards-closed with respect to $\subflat$, such that the statement holds for every $\Gamma\subseteq \Sigma_i$. Choose some $X\in \Sigma\setminus \Sigma_i$ such that if $X'\psubflat X$, then already $X'\in \Sigma_i$. Let $\Sigma_{i+1} = \Sigma_i\cup\set{X}$. 

Assume first that $\dm(X) = 1$. Then $X$ intersects $\bigcup \Sigma_i$, and every element of $\Sigma_i$, in $F$. In particular, $\dm(\bigcap S\cap X) = 0$ for every $\emptyset\neq S\subseteq \Sigma$. By definition, $\dm(X) = |X| - \alpha_G(X) - \alpha_G(F) = |X\setminus F| - \alpha_G(X)$. We compute
\begin{align*}
\flatsum{G}(\Sigma_{i+1}) -\dm(G) &= \flatsum{G}(\Sigma_{i}) + \dm(X) -\dm(G)
\\
&= \left(\sum_{\substack{Y\subflat G\\Y\notin \Sigma_i}} \alpha_G(Y) - \left|G\setminus \bigcup \Sigma_i\right|\right) - (|X\setminus F| -\alpha_G(X))
\\
&= \sum_{\substack{Y\subflat G\\Y\notin \Sigma_{i+1}}} \alpha_G(Y) - \left|G\setminus\bigcup \Sigma_{i+1}\right|
\end{align*}

Now Assume $\dm(X) > 1$. Let $\Sigma_X = \setarg{E\cap X}{E\in \Sigma_i} = \setarg{E}{E\psubflat X}$. Note that $\bigcup \Sigma_X = X\subseteq \bigcup \Sigma_{i}$. Denote $A = \bigcup \Sigma_i = \bigcup \Sigma_{i+1}$. By assumption,
\[
\flatsum{G}(\Sigma_i) -\dm(G) = \sum_{\substack{Y\subflat G\\Y\notin \Sigma_i}} \alpha_G(Y) - |G\setminus A| = \sum_{\substack{Y\subflat G\\Y\notin \Sigma_{i+1}}} \alpha_G(Y) + \alpha_G(X) - |G\setminus A|
\]
and
\begin{align*}
\tag{$*$}
\flatsum{G}(\Sigma_X) &= \left(\sum_{\substack{Y\subflat G\\Y\notin \Sigma_X}} \alpha_G(Y) + \dm(G)\right) - |G\setminus X|
\\
&= \left(|G| - \sum_{\substack{Y\subflat G\\Y\in \Sigma_X}} \alpha_G(Y)\right) - |G\setminus X|
\\
&= |X| - \sum_{Y\psubflat X} \alpha_G(Y) = \alpha_G(X) + \dm(X)
\end{align*}
where the first equality is by induction hypothesis, the second by definition of $\alpha_G(G)$, and the fourth by definition of $\alpha_G(X)$.
Thus,
\begin{align*}
\flatsum{G}(\Sigma_{i+1}) -\dm(G) &= \left(\flatsum{G}(\Sigma_i) + \dm(X) - \flatsum{G}(\Sigma_X)\right) - \dm(G)
\\
&= \left(\flatsum{G}(\Sigma_i) - \dm(G)\right) - \left(\flatsum{G}(\Sigma_X) - \dm(X)\right)
\\
&= \sum_{\substack{Y\subflat G\\Y\notin \Sigma_{i+1}}} \alpha_G(Y) - |G\setminus A|
\end{align*}
\end{proof}

Now that Lemma \ref{lemma: true meaning of alternating sum} is proved, we may apply the equality $(*)$ whenever $\Sigma=\setarg{E}{E\psubflat X}$, for some arbitrary $X$.

\begin{cor}[\hspace{-0.001pt}\cite{EvansMatroid}, Lemma 4.2]
\label{corollary: flatsum of all closed subsets}
Let $X$ be a union of closed sets in a finite pregeometry $G$ and let $\Sigma =\setarg{E}{E\psubflat X}$. Then $\alpha_G(X) = \flatsum{G}(\Sigma) - \dm(X)$.
\end{cor}

\begin{proof}[\textbf{Proof of Proposition \ref{prop: characterization of flatness via alpha}}]
Let $\Sigma$ be some collection of closed sets, denote $X= \bigcup \Sigma$. Assume $X\notin \Sigma$, for otherwise clearly $\flatsum{G}(\Sigma) = \dm(X)$. By Corollary \ref{corollary: adding set to flatsum that is already contained in union}, increasing $\Sigma$ to $\setarg{E}{E\psubflat X}$ does not restrict generality, as it only decreases $\flatsum{G}(\Sigma)$. Then $\alpha_G(X) \geq 0$ if and only if $\flatsum{G}(\Sigma) \geq \dm(\bigcup \Sigma)$.
\end{proof}

\subsection{Geometric prerank}

We now define a notion of \emph{prerank}, which will be to flat pregeometries what $\delta$ is to hypergraphs. Much in the same way that the \emph{predimension} $\delta$ approximates dimension (Morley rank, in Hrushovski's non-collapsed construction), our prerank $\prerank$ will be closely related to Morley rank\footnote{A full analysis of Morley rank is not included in this text. Morley rank is ``shifted'' with respect to $\prerank$, namely $\omega^{-4}\cdot\prerank$, but we find our definition of $\prerank$ more convenient to work with in the context of this paper. See Digression immediately after the proof of Lemma \ref{lemma: edge resolution} for an explanation.} and quantifier elimination in the soon-to-come generic construction.

\begin{definition}
Define $\OrdMod$ to be the free $\Ints$-module generated by $\setcol{\omega^i}{i< \omega+1}$ and endowed with the order where $\sum_{i<\omega+1} a_i\omega^i < \sum_{i<\omega+1} b_i\omega^i$ if and only if $a_j<b_j$, where $j =\max\setcol{i}{a_i\neq b_i}$. That is, the reverse-lexicographical order, i.e., $3\omega^\omega + 7\omega^5 < 4\omega^\omega + 6\omega^5 < 4\omega^\omega + \omega^8$.

When all coefficients of $\alpha, \beta\in \OrdMod$ are non-negative, addition in $\OrdMod$ is precisely the \emph{natural sum} (or \emph{Hessenberg sum}).
\end{definition}

\begin{definition}
For every finite pregeometry $G$ assign $\prerank(G)\in\OrdMod$ by
\[
\prerank(G) = \dm(G)\omega^\omega + \sum_{X\subflat G} \alpha_G(X)\omega^{\dm(X)}
\]
For $H\sqsubseteq^* G$ pregeometries with $H$ finite, write $H\strong_r G$ if $\prerank(H')\geq \prerank(H)$ for every finite intermediate $H\subseteq H'\sqsubseteq^* G$.
\end{definition}

Our goal now is to show that, in the flat context, $\sqsubseteq$ and $\strong_r$ are equivalent. The following two lemmas lead up to Corollary \ref{corollary: sqsubseteq implies strong_r}, the left-to-right implication, and Proposition \ref{proposition: strong_r implies sqsubseteq} is the right-to-left implication.

\begin{lemma}
\label{lemma: closed preserves alpha}
If $H\subseteq G$ is the induced pregeometry on a closed subset of $G$, then $\alpha_H$ is the restriction of $\alpha_G$ to finite subsets of $H$.
\end{lemma}

\begin{proof}
Observe that for any $F\subseteq H$, $F$ is closed in $H$ if and only if it is closed in $G$. We prove inductively. Let $X\subseteq H$ be such that for every $F\psubflat X$ we have $\alpha_H(F) = \alpha_G(F)$. Then
\[
\alpha_G(X) = |X|-\dm(X) - \sum_{Y\psubflat X}\alpha_G(Y) = |X|-\dm(X) - \sum_{Y\psubflat X}\alpha_H(Y) = \alpha_H(X)
\]
\end{proof}

\begin{lemma}
\label{lemma: sqsubseteq means alpha only increases}
Let $H\sqsubseteq G$ be finite pregeometries and let $X\subflat H$. Then $\alpha_H(X) \leq \alpha_G(\cl_G(X))$.
\end{lemma}

\begin{proof}
Denote $Y=\cl_G(X)$ and let $\Sigma = \setarg{E}{E\psubflat Y\subseteq G}$. Observe that whenever $E\in \Sigma$, then $\dm(E) < \dm(Y) = \dm(X)$, so $E\cap X\psubflat X\subseteq H$. Hence, $\Sigma_H = \setarg{F}{F\psubflat X\subseteq H}$. By $H\sqsubseteq G$, we have $\flatsum{H}(\Sigma_H) \leq \flatsum{G}(\Sigma)$. By Lemma \ref{lemma: closed preserves alpha} and Corollary \ref{corollary: flatsum of all closed subsets},
\begin{gather*}
\alpha_G(Y) = \alpha_Y(Y) = \flatsum{Y}(\Sigma) - \dm(Y) = \flatsum{G}(\Sigma) -\dm(Y)
\\
\alpha_H(X) = \alpha_X(X) = \flatsum{X}(\Sigma_H) -\dm(X) = \flatsum{H}(\Sigma_H) -\dm(X)
\end{gather*}
Thus, as $\dm(X) = \dm(Y)$, we get  $\alpha_H(X) \leq \alpha_G(Y)$.
\end{proof}

\begin{cor}
\label{corollary: sqsubseteq implies strong_r}
If $G$ is a flat pregeometry and $H\sqsubseteq G$ is finite, then $H\strong_r G$.
\end{cor}

\begin{proof}
Let $H\subseteq H'\sqsubseteq^* G$ be finite. Then $H'$ is flat and, by Lemma \ref{lemma: sqsubseteq and sqsubseteq^* interaction}, $H\sqsubseteq H'$. By Proposition \ref{prop: characterization of flatness via alpha} Flatness gives that whenever $Y\subflat H'$, then $\alpha_{H'}(Y)\geq 0$. Then, using Lemma \ref{lemma: sqsubseteq means alpha only increases},
\begin{align*}
\prerank(H) &= \dm(H)\omega^\omega + \sum_{X\subflat H}\alpha_H(X)\omega^{\dm(X)}
\\
&\leq \dm(H')\omega^\omega + \sum_{X\subflat H}\alpha_{H'}(\cl_{H'}(X))\omega^{\dm(X)}
\\
&\leq \dm(H')\omega^\omega + \sum_{Y\subflat H'}\alpha_{H'}(Y)\omega^{\dm(Y)} =\prerank(H')
\end{align*}
\end{proof}

\begin{prop}
\label{proposition: strong_r implies sqsubseteq}
If $G$ is a flat pregeometry and $H\sqsubseteq^* G$ is finite such that $H\strong_r G$, then $H\sqsubseteq G$.
\end{prop}

\begin{proof}
It will be enough to show that $H\sqsubseteq G_0$ for some $H\subseteq G_0\sqsubseteq G$, so we may assume there is no such $G_0$ distinct from $G$. In particular, $G$ is finite and $\dm(H) = \dm(G)$.

Let $P,M,R,S, f, g, \hg{A},\hg{B}$ be as in the statement of Proposition \ref{prop: sqsubseteq implies leqslant}, and recall $R[P]\subseteq S$. Let $R' = R\setminus R[P]$, $S' = S\setminus R[P]$, $\hg{A}' = (M, R')$, $\hg{B}' = (P, S')$, $G' = G_{\hg{A}'}$, $H' = G_{\hg{B}'}$. Note that $H'\subseteq G'$.

Let $\Sigma = \setarg{\cl_H(r)}{r\in S'}\cup\setarg{\cl_H(a)}{a\in P}$ and observe that every set that is closed in $H$ is also closed in $H'$. Note that $\delta_{\hg{A}}(X/X\cap P) = \delta_{\hg{A}'}(X/X\cap P)$ for every $X\subseteq M$, hence $\Sigma_G = \Sigma_{G'}$. For every $X$ closed in $G$, it holds that 
\begin{align*}
\dm_{G'}(X) - \dm_{H'}(X\cap P) &= \delta_{\hg{A}'}(X) - \delta_{\hg{B}'}(X\cap P)
\\
&= \delta_{\hg{A}}(X) - \delta_{\hg{B}}(X\cap P) = \dm_{G}(X) - \dm_H(X\cap P)
\end{align*}
so $\flatsum{G'}(\Sigma_G) - \flatsum{H'}(\Sigma) = \flatsum{G}(\Sigma_G) - \flatsum{H}(\Sigma)$. By $H\sqsubseteq^* G$, we have $\flatsum{H}(\Sigma) = \flatsum{G}(\Sigma_G)$, hence $\flatsum{H'}(\Sigma) = \flatsum{G'}(\Sigma_G)$.

By applying the additional part of Proposition \ref{prop: flatness} in $\hg{B}'$, we get $\flatsum{H'}(\Sigma) = \dm_{H'}(P) = \dm_{G'}(P)$. Denote $E = \bigcup \Sigma_G$. By flatness, $\dm_{G'}(E) \leq \flatsum{G'}(\Sigma_G) = \dm_{G'}(P)$. Since $P\subseteq E$, in fact $\dm_{G'}(E) = \dm_{G'}(P)$, so $\dm_{G'}(\bigcup\Sigma_G) = \flatsum{G'}(\Sigma_G)$. Using the additional part of Proposition \ref{prop: flatness}, this time in the other direction, we get $E\strong \hg{A}'$ and $R'[E] = \bigcup_{F\in \Sigma_G}R'[F]$. By construction, since $P\subseteq E$, we have $E \strong \hg{A}$. By the minimality assumption on $G$, this means $E= M$.

Assuming $S'\neq \emptyset$, denote $n:=\max\setcol{|e|}{e\in S'}$.

\medskip
\noindent\textbf{Claim.} $R'\subseteq [M]^{< n}$ 

\begin{proof}[Proof of claim]
Assume to the contrary that there exists $e\in R'\cap [M]^{\geq n}$. We have already seen $R' = \bigcup_{F\in \Sigma_G}R'[F]$, so there is some $F\in \Sigma_G$ such that $e\in R'[F]$. Furthermore, $|e| =n$ and $F=\cl_G(e)$.

Let $R'' = R'\setminus \setcol{e\in R'[F]}{|e|=n}$ and denote $\hg{A}'' = (M, R'')$, $G'' = G_{\hg{A}''}$. Observe that every set closed in $\hg{A}'$ is also closed in $\hg{A}''$. Also observe that for every non-empty $S\subseteq \Sigma_G$ distinct from $\set{F}$, since $\hg{A}'[\bigcap S] = \hg{A}''[\bigcap S]$, we have $\dm_{G'}(\bigcap S) = \dm_{G''}(\bigcap S)$, $\dm_{G'}(P\cap\bigcap S) = \dm_{G''}(P\cap\bigcap S)$. Therefore, $\Sigma_{G'}\setminus\set{F} \subseteq \Sigma_{G''}$.

By construction, because $f$ is centered at $P$, the edges in $R'\setminus R''$ were introduced into $R$ only after every tuple in $F\cap P$ had the correct dimension, so $\dm_{G''}(F\cap P) = \dm_{G'}(F\cap P)=\dm_{G'}(F)$. Letting $F_0 = \cl_{G''}(F\cap P)$, we see $\Sigma_{G''} = \left(\Sigma_{G'}\setminus\set{F}\right)\cup\set{F_0}$.

We claim that for every non-empty $S\subseteq\Sigma_{G'}\setminus \set{F}$, $F\cap \bigcap S = F_0\cap \bigcap S$. Denote $Q:=\bigcap S$. As $H\sqsubseteq^* G$, Lemma \ref{lemma: in sqsubseteq^* the intersection and closure commute} gives that $\cl_G(F\cap Q\cap P) = F\cap Q$. So the same is true in $G'$, i.e., $\cl_{G'}(F\cap Q\cap P) = F\cap Q$. Since $\dm_{G'}(F\cap Q)<n-1$, we have $\cl_{G'}(F\cap Q\cap P) = \cl_{G''}(F\cap Q\cap P)$. Noting $F\cap Q\cap P = F_0\cap Q\cap P$ and that $F_0\cap Q$ is closed in $G''$, we get $F\cap Q\subseteq F_0\cap Q$. The inclusion in the other direction is clear, so we achieve the desired equality. Consequently, since $\hg{A}'[F\cap Q] = \hg{A''}[F\cap Q]$, we have $\dm_{G'}(F\cap Q) = \dm_{G''}(F_0\cap Q)$.

The summands in $\flatsum{G'}(\Sigma_{G'})$ and $\flatsum{G''}(\Sigma_{G''})$ are identical, hence, using flatness,
\[
\dm_{G''}(\bigcup \Sigma_{G''}) \leq \flatsum{G''}(\Sigma_{G''}) = \flatsum{G'}(\Sigma_{G'})=\dm_{G'}(P) \leq \dm_{G''}(P).
\]
Denoting $E_0 = \bigcup \Sigma_{G''}$, since $P\subseteq E_0$, we have that all inequalities above are in fact equalities. Applying the additional part of Proposition \ref{prop: flatness} to the equality $\dm_{G''}(E_0) = \flatsum{G''}(\Sigma_{G''})$, we get $E_0\strong \hg{A}''$. So $\delta_{\hg{A}'}(E_0)\leq \delta_{\hg{A}''}(E_0)= \dm_{G''}(P)=\dm_{G'}(P)$, and again by $P\subseteq E_0$ we have $E_0\strong \hg{A}'$. By construction, because $E_0$ contains $P$, this implies $E_0\strong \hg{A}$.

The minimality assumption on $G$ forces now that $E_0 = M$. But now $\delta(\hg{A}') = \delta_{\hg{A}'}(E_0) = \dm_{G'}(P) = \dm_{G''}(E_0) = \delta(\hg{A}'')$, so $|R''|=|R'|$ in contradiction to $e\in R'\cap [M]^{\geq n}$ existing.
\end{proof}
By Proposition \ref{prop: alpha counts edges}, 
\begin{align*}
&\prerank(H) = \dm_H(P)\omega^\omega + \sum_{r\in S} \omega^{|r|-1}, & \ \prerank(G) = \dm_{G}(M)\omega^\omega + \sum_{r\in R} \omega^{|r|-1} 
\end{align*}
Recall $\dm_G(M) = \dm_H(P)$, then
\begin{align*}
\prerank(H) -\prerank(G) &= \sum_{r\in S} \omega^{|r|-1} - \sum_{r\in R} \omega^{|r|-1}
\\
&= \sum_{r\in S'} \omega^{|r|-1} - \sum_{r\in R'} \omega^{|r|-1}
\end{align*}
By choice of $n$, we know that $\sum_{r\in S'} \omega^{|r|-1}\geq \omega^{n-1}$. By the claim we know $\sum_{r\in R'} \omega^{|r|-1} < \omega^{n-1}$. Therefore, $\prerank(H)>\prerank(G)$, in contradiction to $H\strong_r G$.

Thus, we must have $S'=\emptyset$ and by Proposition \ref{prop: sqsubseteq implies leqslant} we conclude that ${H\sqsubseteq G}$.
\end{proof}

The next Corollary summarizes the first section:

\begin{cor}
\label{corollary: sqsubseteq equivalence}
For a pregeometry $G$ and a finite $H\subseteq G$, the following are equivalent:
\begin{enumerate}
\item
$H\sqsubseteq G$
\item
$H\sqsubseteq^* G$ and $H\strong_r G$
\item
There exist good representations $\hg{A}$ and $\hg{B}$ of $G$ and $H$, respectively, such that $\hg{B}\strong \hg{A}$.
\end{enumerate}
\end{cor}

\section{Generic flat pregeometries}

\subsection{Construction of $\pg{G}$}
\label{subsection: construction of G}

In the context of a class of relational structures $\mathbb{D}$ and $\strong_{\mathbb{D}}$ --- a transitive, invariant-under-isomorphism notion of distinguished substructure between elements of $\mathbb{D}$ --- the following special case of Fra\"iss\'e's Theorem gives a method of constructing a generic structure for a subclass $\mathbb{C}$.

\begin{theorem}
\label{theorem: Fraisse's theorem - restricted}
Let $\mathbb{C}\subseteq \mathbb{D}$ be a countable (up to isomorphism) class of finite structures, closed under isomorphisms and taking $\strong_{\mathbb{D}}$-substructures. Assume
\begin{enumerate}
\item
$\emptyset\strong_{\mathbb{D}} A$, for every $A\in \mathbb{C}$
\item
For every $A,B_1,B_2\in \mathbb{C}$ with embeddings $f_i: A\to B_i$ such that ${f_i[A]\strong_{\mathbb{D}} B_i}$, there exists $D\in \mathbb{C}$ and $g_i:B_i\to D$ such that $g_i[B_i]\strong_{\mathbb{D}} D$ and $g_1f_1=g_2f_2$.
\end{enumerate}
Then there exists a unique (up to isomorphism) countable structure $M$ such that $\mathbb{C} = \setcol{A\strong_{\mathbb{D}} M}{|A| < \infty}$ and
\begin{itemize}
\item[$(*)$]
Whenever $A\strong_{\mathbb{D}} M$, $A\strong_{\mathbb{D}} B\in \mathbb{C}$, then there is an embedding $f:B\to M$ fixing $A$ such that $f[B]\strong_{\mathbb{D}} M$.
\end{itemize}
Call $M$ the generic structure for $\mathbb{C}$.
\end{theorem}

\begin{remark}
\label{remark: isomorphism extends to automorphism in generic}
If $M$ is generic for $\mathbb{C}$, a standard back and forth argument shows that any isomorphism between finite $\strong_{\mathbb{D}}$-embedded substructures extends to an automorphism of $M$.
\end{remark}

The procedure with which Hrushovski's non-collapsed construction is attained is an application of the theorem to a class of hypergraphs. From the properties in Fact \ref{fact: self-sufficiency properties}, it is not hard to show that the conditions of Theorem \ref{theorem: Fraisse's theorem - restricted} hold.

\begin{definition}
\label{definition: Hrushovski's construction}
For every $n\in \Nats$, define $\mathcal{C}_n$ to be the class of finite hypergraphs $\hg{A}=(M,R)$ such that $\emptyset\strong \hg{A}$ and $R\subseteq [M]^{\leq n}$. For $n=\omega$, define $\mathcal{C}_{\omega} = \bigcup_{n\in\Nats} \mathcal{C}_n$. Denote by $\mathcal{M}_n$ the generic structure for $\mathcal{C}_n$.
\end{definition}

When this causes no confusion, we omit the subscript and write by convention $\mathcal{C}$, $\mathcal{M}$ for $\mathcal{C}_{\omega}$, $\mathcal{M}_{\omega}$, respectively.

Model theoretically, the structures $\mathcal{M}_n$ are saturated, $\omega$-stable, and almost model complete (have quantifier elimination up to boolean combinations of existential formulas). We will similarly construct generic flat pregeometries $\pg{G}_n$, sharing similar traits, and demonstrate that in fact $\pg{G}_n = G_{\mathcal{M}_n}$. Since the procedure goes through regardless of arity, we do the work with unbounded arity. We geometrically define and address bounded arities in a later subsection.

\begin{definition}
Define $\mathfrak{C}$ to be the class of all finite flat pregeometries.
\end{definition}

By Proposition \ref{theorem: flat pregeometry comes from graph}, in fact $\mathfrak{C} = \setarg{G_{\hg{A}}}{\hg{A}\in \mathcal{C}}$. By Corollary \ref{corollary: sqsubseteq equivalence}, it is clear that $G$ is flat if and only if $\emptyset\sqsubseteq G$. In order to apply Theorem \ref{theorem: Fraisse's theorem - restricted}, we only need to show amalgamation. We will go through hypergraphs to do this.

\begin{definition}
\label{definition: hypergraph amalg}
Let $\hg{B}_1 = (M_1, R_1)$, $\hg{B}_2 = (M_2, R_2)$ be hypergraphs such that $\emptyset\strong \hg{B}_1, \hg{B}_2$. Denote $M_0 = M_1\cap M_2$ and assume $M_0\strong \hg{B}_1$. Define the hypergraph
\[
\hg{B}_1\amalg\hg{B}_2 = (M_1\cup M_2,\ R_2\cup (R_1\setminus R_1[M_0]))
\]
\end{definition}

While this is not necessarily an amalgam of hypergraphs, if both $\hg{B}_1$ and $\hg{B}_2$ induce the same pregeometry on their intersection, $G_{\hg{B}_1\amalg\hg{B}_2}$ will be an amalgam of pregeometries. Before stating the definition of the amalgam for pregeometries in Definition \ref{definition: geometric amalgam}, we first show it is well-defined in Corollary \ref{corollary: geometric amalgam is well-defined}, and capitalize on that to get a short useful result in Corollary \ref{corollary: replace self-sufficient subgraph}. The following Lemma demonstrates that the amalgam is a ``free'' amalgam.

\begin{lemma}
\label{lemma: dimension in amalg}
In the notation of Definition \ref{definition: hypergraph amalg} above, assume that $G_{\hg{B}_1[M_0]}\subseteq G_{\hg{B}_2}$. That is, $G_{\hg{B}_1}$ and $G_{\hg{B}_2}$ restrict to the same pregeometry on $M_0$. Then for any $X\subseteq M_1\cup M_2$ closed in $\hg{D}:=\hg{B}_1\amalg\hg{B}_2$ of finite dimension,
\[
\dm_{\hg{D}}(X) = \dm_{\hg{B}_2}(X\cap M_2) + \dm_{\hg{B}_1}(X\cap M_1) - \dm_{\hg{B}_1}(X\cap M_0)
\]
\end{lemma}

\begin{proof}
Assume for a moment that $\hg{D}$ is finite. Since $X$ is closed in $\hg{D}$, $\dm_{\hg{D}}(X) = \delta_{\hg{D}}(X)$. By construction, $X\cap M_2$ is clearly closed in $\hg{B}_2$, so 
\[
\dm_{\hg{B}_2}(X\cap M_2) = \delta_{\hg{B}_2}(X\cap M_2) = \delta_{\hg{D}}(X\cap M_2).
\]
Also, $X\cap M_1$ is closed in $\hg{B}_1$ and since $M_0\strong \hg{B}_1$, the set $X\cap M_0$ is closed in $\hg{B}_1[M_0]$. So
\begin{align*}
\delta_{\hg{D}}(X\cap M_1/X\cap M_0) &= \delta_{\hg{B}_1}(X\cap M_1/X\cap M_0)
\\
&= \delta_{\hg{B}_1}(X\cap M_1) - \delta_{\hg{B}_1}(X\cap M_0)
\\
&= \dm_{\hg{B}_1}(X\cap M_1) - \dm_{\hg{B}_1}(X\cap M_0)
\end{align*}
By construction, $\delta_{\hg{D}}(X/X\cap M_2) = \delta_{\hg{D}}(X\cap M_1/X\cap M_0)$. Since, by definition, $\delta_{\hg{D}}(X) = \delta_{\hg{D}}(X\cap M_2) + \delta_{\hg{D}}(X/X\cap M_2)$, we are done.

Now, if $\hg{D}$ is infinite, reduce to a self-sufficient subgraph of $\hg{D}$ containing bases for $X$, $X\cap M_1$, $X\cap M_2$, and $X\cap M_0$ for the argument to go through.
\end{proof}

\begin{cor}
\label{corollary: geometric amalgam is well-defined}
In the notation of Definition \ref{definition: hypergraph amalg}, assuming $G_{\hg{B}_1[M_0]}\subseteq G_{\hg{B}_2}$, the pregeometry $G:= G_{\hg{B}_1 \amalg \hg{B}_2}$ does not depend on the structure of $\hg{B}_1$ and $\hg{B}_2$, but only on $G_{\hg{B}_1}$ and $G_{\hg{B}_2}$.
\end{cor}

\begin{proof}
Denote $G:=G_\hg{D}$. Let $\hg{B}_1'$, $\hg{B}_2'$ be such that $M_0\strong \hg{B}_1'$ and $G_{\hg{B}_i'} = G_{\hg{B}_i}$ for $i=1,2$. Denote $G' = G_{\hg{B}_1'\amalg \hg{B}_2'}$. Let $X\subseteq M_1\cup M_2$ be closed in $G'$. Then as before $X\cap M_i$ is closed in $\hg{B}_i'$ for $i=1,2$, $X\cap M_0$ is closed in $\hg{B}_1'[M_0]$, and so
\begin{align*}
\dm_{G'}(X) &= \dm_{\hg{B}_2'}(X\cap M_2) + \dm_{\hg{B}_1'}(X\cap M_1) - \dm_{\hg{B}_1'}(X\cap M_0)
\\
&= \dm_{\hg{B}_2}(X\cap M_2) + \dm_{\hg{B}_1}(X\cap M_1) - \dm_{\hg{B}_1}(X\cap M_0)
\\
&= \delta_{\hg{B}_2}(X\cap M_2) + \delta_{\hg{B}_1}(X\cap M_1/X\cap M_0)
\\
&= \delta_{\hg{D}}(X) \geq \dm_{\hg{D}}(X) = \dm_{G}(X)
\end{align*}
By symmetry of the argument, taking $Y=\cl_{G}(X)$, we have $\dm_{G'}(Y)\leq \dm_{G}(Y) = \dm_{G}(X)$. By definition, $\dm_{G'}(X)\leq \dm_{G'}(Y)$ so we get that $\dm_{G'}(Y) = \dm_{G'}(X)$, i.e., $Y=X$. Since $G$ and $G'$ have the exact same closed sets, $G = G'$.
\end{proof} 

\begin{cor}
\label{corollary: replace self-sufficient subgraph}
If $H\sqsubseteq G$ are flat and $\hg{B}$ is a representation of $H$, then there exists $\hg{A}$ a representation of $G$ such that $\hg{B}\strong \hg{A}$.
\end{cor}

\begin{proof}
Choose some representations $\hg{B}'\strong \hg{A}'$ of $H\sqsubseteq G$. Note $\hg{A'} = \hg{A}'\amalg\hg{B}'$. Let $\hg{A} = \hg{A}'\amalg \hg{B}$, clearly $\hg{B}\strong \hg{A}$. Because $G_{\hg{B}}=G_{\hg{B}'}$, we also have $G_{\hg{A}'\amalg \hg{B}} = G_{\hg{A}'\amalg \hg{B}'}$, so $G_{\hg{A}} = G_{\hg{A}'} = G$.
\end{proof}

We can now rigorously define the geometric amalgam and show that it indeed (strongly) extends the component pregeometries.

\begin{definition}
\label{definition: geometric amalgam}
For flat pregeometries $H, G_1, G_2$ such that $H\sqsubseteq G_1$, $H\subseteq G_2$ and $H=G_1\cap G_2$. We define $G_1\amalg_{H} G_2$, the \emph{amalgam of $G_1$ and $G_2$ over $H$}, to be the pregeometry associated to $\hg{B}_1\amalg \hg{B}_2$, where $\hg{B}_i$ is a good representation of $G_i$ and $\hg{B}_1[H]\strong \hg{B}_1$. 
\end{definition}

\begin{lemma}
\label{lemma: amalg properties}
In the notation of Definition \ref{definition: geometric amalgam}, letting $G := G_1\amalg_H G_2$,
\begin{enumerate}[i.]
\item
$G_2\sqsubseteq G$
\item
$G_1 \subseteq G$
\item
Whenever $X$ is a closed set in $G_i$, then $\cl_G(X) = X\cup cl_{G_{3-i}}(X\cap H)$
\item
If $H\sqsubseteq^* G_2$, then $G_1 \sqsubseteq^* G$
\item
If $H\sqsubseteq G_2$, then $G_1 \sqsubseteq G$
\end{enumerate}

\end{lemma}

\begin{proof}
Let $\hg{B}_1 = (M_1, R_1)$, $\hg{B}_2=(M_2,R_2)$ be good representation of $G_1$, $G_2$ such that $M_0:=M_1\cap M_2\strong \hg{B}_1$, and let $\hg{D} = \hg{B}_1\amalg \hg{B}_2$.

\smallskip
i. As $\hg{B}_2\strong D$, it is clear that $G_2\sqsubseteq G$.

\smallskip
ii. Let $X\subseteq M_1$. Letting $Y=\cl_{G}(X)$, from Lemma \ref{lemma: dimension in amalg} we have that $\dm_{G}(Y) \geq \dm_{\hg{B}_1}(Y\cap M_1) \geq \dm_{\hg{B}_1}(X)$. So $\dm_G(X) \geq \dm_{\hg{B}_1}(X)$. Now let $Z_1 = \cl_{\hg{B}_1}(X)$, $Z_2 = \cl_{\hg{B}_2}(Z_1\cap M_0)$, and $Z = Z_1\cup Z_2$. Observe that $Z_2\cap M_0 = Z_1\cap M_0$. Then
\begin{align*}
\delta_{D}(Z) &= \delta_{\hg{B}_2}(Z_2) + \delta_{\hg{B}_1}(Z_1/Z_1\cap M_0)
\\
&= \dm_{\hg{B}_2}(Z_1\cap M_0) + \dm_{\hg{B}_1}(Z_1) - \dm_{\hg{B}_1}(Z_1\cap M_0)
= \dm_{\hg{B}_1}(Z_1).
\end{align*}
So $\dm_G(X) \leq \delta_{\hg{D}}(Z) = \dm_{\hg{B}_1}(Z_1) = \dm_{\hg{B}_1}(X)$. Conclude $\dm_G(X) = \dm_{\hg{B}_1}(X)$.

\smallskip
iii. Fix $i\in\set{1,2}$. Let $X$ be closed in $G_i$ and let $Y=\cl_G(X)$. By Lemma \ref{lemma: dimension in amalg},
\[
\dm(Y) = \dm(Y\cap M_i) + \dm(Y\cap M_{3-i}) - \dm(Y\cap M_0).
\]
Since $\dm(Y)= \dm(X) = \dm(Y\cap M_i)$, we get that $\dm(Y\cap M_{3-i}) = \dm(Y\cap M_0)$, i.e., $Y\cap M_{3-i} = \cl_{G_{3-i}}(Y\cap M_0)$. As $Y\cap M_0 = X\cap M_0$, we have
\[
Y = (Y\cap M_i) \cup (Y\cap M_{3-i}) = X\cup \cl_{G_{3-i}}(X\cap M_0).\]

\smallskip
iv. Assume $H\sqsubseteq^* G_2$. Let $X_1,X_2$ be closed in $G_1$, and denote $Y_i=\cl_G(X_i)$. Then by the previous item,
\begin{align*}
Y_1\cap Y_2 &= (M_1\cap Y_1\cap Y_2)\cup (M_2\cap Y_1\cap Y_2)
\\
&= (X_1\cap X_2)\cup (\cl_{G_2}(X_1\cap M_0)\cap \cl_{G_2}(X_2\cap M_0))
\\
&= (X_1\cap X_2) \cup (\cl_{G_2}(X_1\cap X_2\cap M_0)) = \cl_G(X_1\cap X_2)
\end{align*}
where going from the second line to the third is by $H\sqsubseteq^* G_2$. Then $\dm(X_1\cap X_2) = \dm(Y_1\cap Y_2)$ and $G_1\sqsubseteq^* G$.

\smallskip
v. Assume $H\sqsubseteq G_2$. By Corollary \ref{corollary: replace self-sufficient subgraph} we may assume that $M_0\strong \hg{B}_2$ and $\hg{B}_1[M_0] = \hg{B}_2[M_0]$. So by construction $\hg{B}_1\strong \hg{D}$ and $G_1 \sqsubseteq G$.
\end{proof}

\begin{cor}
\label{cor: amalgam geometric summary}
Let $H\sqsubseteq G_1,G_2$ be flat pregeometries with $H=G_1\cap G_2$. Then there exists a flat pregeometry $G$ such that $G_1,G_2\sqsubseteq G$. Moreover, $G$ is defined on the union of the sets on which $G_1$, $G_2$ are defined. \qed
\end{cor}

We have proven that $\mathfrak{C}$ is an amalgamation class. We denote by $\pg{G}$ the countable generic structure guaranteed by Theorem \ref{theorem: Fraisse's theorem - restricted}. We dub $\pg{G}$ the \emph{generic flat pregeometry of unbounded arity}.

Now that we have constructed $\pg{G}$ independently, we show that it is in fact the pregeometry of Hrushovski's non-collapsed construction for hypergraphs of unbounded arity.

\begin{prop}
\label{proposition: G is isomorphic to G_M}
$G_{\mathcal{M}} \cong \pg{G}$.
\end{prop}

\begin{proof}
Since $\mathcal{M}$ is a hypergraph, $G_{\mathcal{M}}$ is flat, and so whenever $H\sqsubseteq G_{\mathcal{M}}$, also $H$ is flat. Thus, $\setcol{H\sqsubseteq G_{\mathcal{M}}}{|H|<\infty}\subseteq \mathfrak{C}$. Conversely, if $H\in \mathfrak{C}$, then for $\hg{A}$ a good representation of $H$, $\hg{A}\in \mathcal{C}$. Without loss of generality, we may assume $\hg{A}\strong \mathcal{M}$ and so $G_{\hg{A}}\sqsubseteq G_{\mathcal{M}}$.

We show that $(*)$ of Theorem \ref{theorem: Fraisse's theorem - restricted} holds with respect to $G_{\mathcal{M}}$, $\mathfrak{C}$ and $\sqsubseteq$. Assume $F\sqsubseteq G_{\mathcal{M}}$ and $F\sqsubseteq H\in \mathfrak{C}$. Let $B=\sscl_{\mathcal{M}}(X_F)$ where $X_F$ is the underlying set of $F$. Denote $K = G_{\mathcal{M}[B]}$. Then $K\sqsubseteq G_{\mathcal{M}}$ and by Lemma \ref{lemma: sqsubseteq and sqsubseteq^* interaction}, also $F\sqsubseteq K$. By renaming elements of $H$, we may assume $H\cap K = F$. Let $L= K\amalg_{F} H$. Then $K\sqsubseteq L$ and so by Corollary \ref{corollary: replace self-sufficient subgraph} choose some $\hg{D}$ a good representation of $L$ such that $\hg{M}[B]\strong \hg{D}$. By genericity of $\mathcal{M}$, we may strongly embed $\hg{D}$ into $\mathcal{M}$ over $\hg{M}[B]$, so without loss of generality $\hg{D}\strong \mathcal{M}$. Now $L\sqsubseteq G_{\mathcal{M}}$ and $F\sqsubseteq H \sqsubseteq L$. Hence, we have strongly embedded $H$ into $G_{\mathcal{M}}$ over $F$.
\end{proof}

\subsection{Model theory of $\pg{G}$}
\label{subsection: Model theory of G}

We now examine pregeometries as first order objects. Fix the language $\mathcal{L} = \setcol{I_n}{n\in \omega}$ where $I_n$ is an $n$-ary relation symbol. We consider a pregeometry $G$ as an $\mathcal{L}$-structure by interpreting $I_n$ as the set of independent $n$-tuples in $G$.

\begin{observation}
\label{observation: first order pregeometry}
Let $G$ be a pregeometry. Then
\begin{enumerate}
\item
If $X$ is definable in $G$ and $n\in \Nats$, then ``$\dm(X) \geq n$'' is an $\mathcal{L}$-formula. Hence ``$\dm(X) = n$'' is an $\mathcal{L}$-formula.
\item
For points $a_1,\dots, a_n$, the set $\cl(a_1\dots a_n)$ is definable.
\item
Using the first two items, for fixed $m,d\in \Nats$, we can quantify over $m$ closed sets of dimension at most $d$, and speak of the dimension of their intersections and unions.
\item
For each $n$, there is an $\mathcal{L}$-formula $\varphi_n(x_1,\dots, x_n)$ stating $\set{x_1,\dots, x_n}\sqsubseteq^* G$.
\end{enumerate}
The class of pregeometries/matroids is an elementary class of $\mathcal{L}$-structures. From the definition of flatness (Definition \ref{def: flatness}), we see that the class of \emph{flat} pregeometries is also an elementary class, given by an infinite scheme of axioms.
\end{observation}

We set out to axiomatize the theory of $\pg{G}$. The axiomatization is similar to Hrushovski's first order axiomatization of his construction \cite{Hns}. The genericity is expressed by a scheme of axioms that, paraphrased to invoke the definition of continuity, state that to achieve an embedding of an extension $H$ over the base $F$ that is at most $\epsilon$ away from being strong, the base $F$ needs to be at most $\delta$ away from being strong in the ambient structure, where $\delta$ depends only on $\epsilon$ and $|H|$. This is (T3) of Definition \ref{definition: theory of generic pregeometry}.

\begin{definition}
Write $X\sqsubseteq^n G$ to mean that $X\sqsubseteq^* G$ and whenever $Y\subseteq G$ contains $X$ such that $Y\sqsubseteq^* G$ and $|Y\setminus X|\leq n$, then $X\sqsubseteq Y$. 
\end{definition}

\begin{observation}
In an ambient structure $G$, ``$X\sqsubseteq^n G$'' is a first order formula in $|X|$ many variables, for $X$ finite of a fixed size.

Observe further that $X\sqsubseteq^n G$ holds for arbitrarily large $n$ if and only if $X\sqsubseteq G$. From left to right, see this equivalence by choosing some finite $Y\sqsubseteq G$ containing $X$ and using that $X\sqsubseteq^{n} G$, where $n\geq |Y\setminus X|$, to get $X\sqsubseteq Y\sqsubseteq G$. From right to left, the implication is immediate by Lemma \ref{lemma: sqsubseteq and sqsubseteq^* interaction}.
\end{observation}

\begin{definition}
\label{definition: theory of generic pregeometry}
For a fixed $\tau:\Nats\to \Nats$ with $\tau(n)\geq n$, let $T_\tau$ be the $\mathcal{L}$-theory stating (in an ambient structure $G$):
\begin{enumerate}[(T1)]
\item
$G$ is a flat pregeometry.
\item
$G$ is infinite-dimensional.
\item
Suppose $F\sqsubseteq^* G$, $F\sqsubseteq H\in\mathfrak{C}$. Then for every natural $n$, if $F\sqsubseteq^{\tau(|H|+n)} G$, then there exists an embedding $f:H\to G$ fixing $F$ such that $f[H]\sqsubseteq^n G$.
\end{enumerate}
\end{definition}

The definition of $T_{\tau}$ a priori depends on the choice of $\tau$. We will see that if $T_{\tau}$ is at all satisfiable, then it implies a complete theory independent of the choice of $\tau$, namely the theory of $\pg{G}$. We prove a series of lemmas, geometric analogues of hypergraph trivialities, to show that a good $\tau$ exists. 

We observe that the operation of amalgamating from the left with a fixed pregeometry preserves the $\sqsubseteq^*$ and $\sqsubseteq$ relations.

\begin{observation} 
\label{observation: intermediate amalgam}
Let $F\sqsubseteq G_1$, $F\subseteq H\sqsubseteq^* G_2$ be pregeometries such that $G_1, G_2$ are flat and $F= G_1\cap G_2$. Then $G_1\amalg_F H \sqsubseteq^* G_1\amalg_F G_2$. Furthermore, if $H\sqsubseteq G_2$, then $G_1\amalg_F H \sqsubseteq G_1\amalg_F G_2$.
\end{observation}

\begin{proof}
Since $H\sqsubseteq^* G_2$ and $F\sqsubseteq G_1$, both $F$ and $H$ are flat. Let $\hg{B}_1\strong \hg{A}_1$ be hypergraphs representing $F\sqsubseteq G_1$ and let $\hg{D}$, $\hg{A}_2$ be hypergraphs representing $H$, $G_2$. Note that $\hg{D}\strong \hg{A}_1\amalg \hg{D}$ and observe that, by definition, $\hg{A}_1\amalg \hg{A}_2 = (\hg{A}_1\amalg \hg{D})\amalg \hg{A}_2$. Thus, $G_1\amalg_F G_2 = (G_1\amalg_F H)\amalg_H G_2$. Lemma \ref{lemma: amalg properties} finishes the proof.
\end{proof}

%

Corollary \ref{corollary: strong in amalgam means strong in G_2} and its preceding lemma describe how $\sqsubseteq^*$/$\sqsubseteq$-embeddedness of a set $K$ in an amalgam reflects on its intersection with each component, given that $K$ contains the base of the amalgam.

\begin{lemma}
\label{lemma: one-sided sigma behaves the same on full amalgam}
Let $G=G_1\amalg_F G_2$ be an amalgam of flat pregeometries. Let $K\subseteq G$ contain $F$ and, for $i\in \set{1,2}$, denote $K_i = K\cap G_i$. Then whenever $\Sigma$ is a finite collection of finite dimensional closed sets in $G_2$, denoting $\Sigma_K =\setarg{\cl_G(E)\cap K}{E\in \Sigma}$, the equality $\flatsum{K}(\Sigma_K) = \flatsum{K_2}(\Sigma_{K_2})$ holds.

Furthermore, if $F\sqsubseteq^* G_2$, whenever $\Sigma$ is a finite collection of finite dimensional closed sets in $G_1$, then $\flatsum{K}(\Sigma_K) = \flatsum{K_1}(\Sigma_{K_1})$.
\end{lemma}

\begin{proof}
Enumerate $\Sigma = \set{E_1,\dots, E_k}$ and recall the notation $E_s = \bigcap_{i\in s} E_i$ for $\emptyset \neq s\subseteq[k]$. By $G_2\sqsubseteq^* G$, we have $\bigcap_{i\in s} \cl_G(E_i) = \cl_G(E_s)$. By the third item of Lemma \ref{lemma: amalg properties}, we have
\begin{align*}
\cl_G(E_s)\cap K  &= (\cl_{G_1}(E_s\cap F)\cup E_s)\cap (K_1\cup K_2)
\\
& = cl_{K_1}(E_s\cap F)\cup (E_s\cap K_2)
\\
&\subseteq \cl_K(E_s\cap K_2).
\end{align*}
Therefore, $\dm(\bigcap_{i\in s} \cl_G(E_i)\cap K) = \dm(E_s\cap K_2)$ for every $\emptyset\neq s\subseteq [k]$, proving $\flatsum{K}(\Sigma_K) = \flatsum{K_2}(\Sigma_{K_2})$.

For the additional part, by the fourth item of Lemma \ref{lemma: amalg properties}, $F\sqsubseteq^* G_2$ implies $G_1\sqsubseteq^* G$, which enables a symmetric argument.
\end{proof}

\begin{cor}
\label{corollary: strong in amalgam means strong in G_2}
Let $G:=G_1\amalg_F G_2$ be an amalgam of flat pregeometries and let $F\subseteq K\subseteq G$. Then
\begin{enumerate}
\item
If $K\sqsubseteq^* G$, then $K_2\sqsubseteq^* G_2$. Moreover, if $K\sqsubseteq G$ then $K_2\sqsubseteq G_2$.
\item
Assuming $F\sqsubseteq^* G_2$, if $K\sqsubseteq^* G$, then $K_1\sqsubseteq^* G_1$. Moreover, under the same assumption, if $K\sqsubseteq G$ then $K_1\sqsubseteq G_1$.
\end{enumerate}
\end{cor}

\begin{proof}
Observe that whenever $\Sigma$ is a finite collection of finite dimensional closed sets in $G_2$, as $G_2\sqsubseteq^* G$ the equality $\flatsum{G_2}(\Sigma) = \flatsum{G}(\Sigma_G)$ holds. Together with Lemma \ref{lemma: one-sided sigma behaves the same on full amalgam}, both parts of (1) are immediate. Item (2) is the same, by the additional part of Lemma \ref{lemma: one-sided sigma behaves the same on full amalgam}.
\end{proof}

The next two lemmas are an analogue to the fact that in an ambient hypergraph $\hg{A}$, if $\hg{D}\subseteq \hg{B}\strong \hg{A}$ and $\hg{D}\not\strong \hg{C}\subseteq \hg{A}$, then $D\not\strong \hg{B}\cap \hg{C}$.

\begin{lemma}
\label{lemma: bound on sqsubseteq^* superset}
There exists a fixed function $g:\Nats\to \Nats$ such that if $G$ is flat and $F\subseteq G$, then there is some $H\sqsubseteq^* G$ containing $F$ with $|H|\leq g(|F|)$.
\end{lemma}

\begin{proof}
Assume the contrary. Let $n$ be such that for every $k\in \Nats$ there exists $G_k$ flat with some $F\in [G_k]^n$ such that whenever $H\sqsubseteq^* G_k$ contains $F$, then $|H|> k$. Adding to $\mathcal{L}$ constant symbols $c_1,\dots, c_n$, consider the theory stating (see Observation \ref{observation: first order pregeometry}) that the ambient structure $G$ is a flat pregeometry and that for every natural $k$
\[
\forall x_1,\dots, x_k\ \set{c_1,\dots, c_n,x_1,\dots, x_k}\not\sqsubseteq^* G.
\]
Then this theory is finitely satisfiable and has a model $G$. Since $G$ is flat, there is a hypergraph $\hg{A}$ representing it. Let $X = \set{c_1^G,\dots, c_n^G}$. Then $\sscl_{\hg{A}}(X)\sqsubseteq G$, so in particular $\sscl_{\hg{A}}(X)\sqsubseteq^* G$. But $\sscl_{\hg{A}}(X)$ is finite, a contradiction.
\end{proof}

\begin{lemma}
\label{lemma: not sqsubseteq is witnessed inside any sqsubseteq superstructure}
There exists a fixed function $h$ such that if $G$ is a flat pregeometry and
\begin{enumerate}
\item
$X\subseteq Y\sqsubseteq G$
\item
$X\not\sqsubseteq Z\sqsubseteq^* G$
\end{enumerate}
then there is some $V\sqsubseteq^* Y$ containing $X$ such that $X\not\sqsubseteq V$ and $|V|\leq h(|Z|)$.
\end{lemma}

\begin{proof}
Fix $g$ as in the statement of Lemma \ref{lemma: bound on sqsubseteq^* superset}. For each natural number $n$, set $h(n) = \max\setarg{g(k)}{k\leq 2^n n}$.

Let $\Sigma_Z$ witness that $X\not\sqsubseteq Z$ and denote $\Sigma = \setarg{\cl_G(E)}{E\in \Sigma_Z}$. Since $Z\sqsubseteq^* G$, we have $\flatsum{Z}(\Sigma_Z) = \flatsum{G}(\Sigma)$. By $Y\sqsubseteq G$, we know $\flatsum{Y}(\Sigma_Y) \leq \flatsum{G}(\Sigma)$. By choice of $\Sigma$, $\flatsum{Z}(\Sigma_Z)  < \flatsum{X}(\Sigma_X)$. Conclude $\flatsum{Y}(\Sigma_Y) < \flatsum{X}(\Sigma_X)$, so $\Sigma_Y$ witnesses that $X\not\sqsubseteq Y$.

There are at most $2^{|Z|}$ elements in $\Sigma$ and each is of dimension at most $\dm(Z)$, so there is some $V'\subseteq Y$ of size at most $2^{|Z|}\dm(Z)$ such that for every $E\in \Sigma$, $\dm(V'\cap E) = \dm(Y\cap E)$.
Now, take a minimal $V\sqsubseteq^* Y$ containing $V'$. Then $|V| \leq g(|V'|)\leq h(|Z|)$ and $\Sigma_V$ witnesses $X\not\sqsubseteq V$.
\end{proof}

\begin{remark}
In Lemma \ref{lemma: bound on sqsubseteq^* superset} and Lemma \ref{lemma: not sqsubseteq is witnessed inside any sqsubseteq superstructure}, there is no harm in assuming the functions $g$ and $h$ are non-decreasing.
\end{remark}

We now have the components required for the proof.

\begin{prop}
\label{proposition: T is satisfied by pg{G}}
There exists $\tau:\Nats\to \Nats$ such that $\pg{G}\models T_{\tau}$. In particular, $T_{\tau}$ is satisfiable.
\end{prop}

\begin{proof}
We only need to address T3.

Fix some $n\geq |H|$. Let $F\sqsubseteq^* \pg{G}$, $F\sqsubseteq H\in\mathfrak{C}$ and let $E\sqsubseteq \pg{G}$ contain $F$. By genericity of $\pg{G}$, we may assume $H\amalg_F E$ is strongly embedded into $\pg{G}$ over $E$.

Assume now that $H\not\sqsubseteq^n \pg{G}$. Let $Z\sqsubseteq^* \pg{G}$, $|Z\setminus H|\leq n$ contain $H$ such that $H\not\sqsubseteq Z$. By Lemma \ref{lemma: not sqsubseteq is witnessed inside any sqsubseteq superstructure}, there is some $V\sqsubseteq^* H\amalg_F E$ containing $H$ of size at most $h(|Z|)\leq h(|H|+n)$ such that $H\not\sqsubseteq V$. By (1) of Corollary \ref{corollary: strong in amalgam means strong in G_2}, we have $V_0:= V\cap E \sqsubseteq^* E$. Then by Observation \ref{observation: intermediate amalgam}, $V= H\amalg_F V_0$. If $F\sqsubseteq V_0$, the last item of Lemma \ref{lemma: amalg properties} implies $H\sqsubseteq V$, which is not the case. So $F\not\sqsubseteq V_0$, where $V_0\sqsubseteq^* \pg{G}$ with $|V_0|\leq h(|H|+n)$.

Conclude that for choosing $\tau$ greater or equal to the $h$ of Lemma \ref{lemma: not sqsubseteq is witnessed inside any sqsubseteq superstructure}, ${\pg{G}\models T_{\tau}}$.
\end{proof}

\begin{notation}
From now on, fix $T:=T_{\tau}$ for some $\tau$ such that $T_{\tau}$ is satisfiable (not necessarily the $\tau$ of the above lemma).
\end{notation}
We will show that $\pg{G}$ is a saturated model for $T$. In the case of Hrushovski's construction, it is easy to show that $\mathcal{M}$ is saturated, since $\mathcal{M}$ is isomorphic to each of its elementary extensions. This is not the case for $\pg{G}$. We instead use the weaker property stated in Proposition \ref{proposition: model of T can be extended to generic}. The proposition is proved by constructing an increasing chain, with each step constructed using the next lemma.


\begin{lemma}
\label{lemma: realize extension strongly}
Let $L\models T$, $F\sqsubseteq L$, $F\sqsubseteq H\in \mathfrak{C}$. Then there exists some elementary extension $L\preceq L'$ and $f:H\to L'$ fixing $F$ such that $f[H]\sqsubseteq L'$.
\end{lemma}

\begin{proof}

Observe that although the language $\mathcal{L}$ is infinite, the atomic type of a finite pregeometry is given by a single finite formula. Denote by $F(\bar{x})$ the atomic type of $F$. Denote by $H(\bar{x}\bar{y})$ the atomic type of $H$, where the induced structure on $\bar{x}$ is that of $F$, and the elements of $\bar{y}$ realize over $\bar{x}$ the atomic type of $H$ over $F$. Denote by $``\bar{x}\sqsubseteq^n G"$ a first order formula in variables $\bar{x}$ stating that $\bar{x}$ is $\sqsubseteq^n$-embedded in the ambient structure. Denote by $``\bar{x}\sqsubseteq G"$ the partial type $\setarg{``\bar{x}\sqsubseteq^n G"}{n\in \omega}$.

Observe that by the axiom scheme T3, for any fixed natural $n$
\[
T\cup \set{F(\bar{x})}\cup ``\bar{x}\sqsubseteq G" \models \exists \bar{y} (H(\bar{x}\bar{y})\wedge ``\bar{x}\bar{y}\sqsubseteq^{n} G").
\]
Whenever $i>j$, $\exists \bar{y} (H(\bar{x}\bar{y})\wedge ``\bar{x}\bar{y}\sqsubseteq^{i} G")\implies \exists \bar{y} (H(\bar{x}\bar{y})\wedge ``\bar{x}\bar{y}\sqsubseteq^{j} G")$. Consequently, if $G\models T\cup \set{F(\bar{a})}\cup ``\bar{a}\sqsubseteq G"$, then the type $\set{H(\bar{a}\bar{y})}\cup ``\bar{a}\bar{y}\sqsubseteq G"$ over $\bar{a}$ is finitely satisfiable.
Since $L\models T$ and $F\sqsubseteq L$, we may realize $H(F\bar{y})\wedge ``F\bar{y}\sqsubseteq G"$ in some elementary extension $L\preceq L'$. This finishes the proof. 
\end{proof}

\begin{prop}
\label{proposition: model of T can be extended to generic}
Whenever $L\models T$ is countable, there exists an elementary extension $L\preceq G$ that is generic for $\mathfrak{C}$. In particular $G\cong \pg{G}$.
\end{prop}

\begin{proof}
We construct an elementary chain, similarly to a Fra\"iss\'e construction, but starting from a model and realizing types instead of amalgamating.

Let $M_0 = L$. Assume $M_i$ countable, $F_i\sqsubseteq M_i$ and $F\sqsubseteq H_i\in \mathfrak{C}$ are given. Use Lemma \ref{lemma: realize extension strongly} to get a countable elementary extension $M_i\preceq M_{i+1}$ into which $H$ can be strongly embedded over $F$. Choose an enumeration so that for every $i<\omega$, every $A\sqsubseteq M_i$ and $A\sqsubseteq B\in\mathfrak{C}$, the pair $(A,B)$ is chosen as $(F_i, H_i)$ infinitely often.

Let $G = \bigcup_{i<\omega} M_i$. Then $L\preceq G$. Observe that being $\sqsubseteq$-embedded in a model is a first order property (type) preserved by elementary extension, so by construction $G$ is clearly generic for $\mathfrak{C}$.
\end{proof}

Before proceeding with the proof of saturation, we note an immediate corollary of Proposition \ref{proposition: model of T can be extended to generic}.

\begin{cor}
$T\equiv Th(\pg{G})$. In particular, $T=T_{\tau}$ is a complete theory independent of the specific choice of $\tau$. \qed
\end{cor}

\begin{theorem}
$\pg{G}$ is saturated.
\end{theorem}

\begin{proof}
Let $p(x)$ be a complete type over a finite set $F\subseteq \pg{G}$. By increasing $F$, we may assume $F\sqsubseteq \pg{G}$. Let $L$ be an elementary extension of $\pg{G}$ in which $p(x)$ is realized, say by $a\in L$. Let $B\sqsubseteq L$ contain $F\cup\set{a}$. Since $F\sqsubseteq L$, in particular $F\sqsubseteq B$. By genericity of $\pg{G}$, we may strongly embed $B$ into $\pg{G}$ over $F$, call the image of this embedding $H\sqsubseteq \pg{G}$.

By Proposition \ref{proposition: model of T can be extended to generic}, let $G$ be an elementary extension of $L$ such that $G\cong \pg{G}$. Then $H,B\sqsubseteq G$ and $H$, $B$ are isomorphic over $F$. As in Remark \ref{remark: isomorphism extends to automorphism in generic}, there exists an automorphism $f$ of $G$ extending the isomorphism between $B$ and $H$. In particular, $\tp^{\pg{G}}(f(a)/F) = \tp^{G}(f(a)/F) = \tp^{G}(a/F) =\tp^{L}(a/F) = p(x)$. So the arbitrary type $p(x)$ is realized in $\pg{G}$, hence $\pg{G}$ is saturated.
\end{proof}

Now that we have saturation of $\pg{G}$, we can show that $T$ is $\omega$-stable and has quantifier elimination up to a set of formulas, reminiscent of the case of $Th(\mathcal{M})$. We lead with $\omega$-stability. The next lemma shows that the type of a strongly embedded set in a model of $T$ is determined by its atomic diagram.

\begin{lemma}
\label{lemma: atomic type determines full type for sqsubseteq}
Let $H,F\sqsubseteq \pg{G}$ (possibly infinite) be such that $f:H\to F$ is an isomorphism of pregeometries. Then, seen as $\omega$-tuples, $\tp^{\pg{G}}(H) = \tp^{\pg{G}}(F)$.
\end{lemma}

\begin{proof}
Let $X\subfin H$ be arbitrary. It is enough to show that $\tp^{\pg{G}}(X) = \tp^{\pg{G}}(f[X])$.  
Let $Y\sqsubseteq H\sqsubseteq \pg{G}$ be finite containing $X$. Then $f[Y]\sqsubseteq f[H] = F$, and since $F\sqsubseteq \pg{G}$, also $f[Y]\sqsubseteq \pg{G}$. Then the restriction of $f$ to elements of $Y$ is an isomorphism between $\sqsubseteq$-embedded finite substructures of $\pg{G}$. As in Remark \ref{remark: isomorphism extends to automorphism in generic}, there is an automorphism of $\pg{G}$ taking $Y$ to $f[Y]$, hence they have the same type. Since $X\subseteq Y$, also $\tp^{\pg{G}}(X) = \tp^{\pg{G}}(f[X])$.
\end{proof}

\begin{theorem}
$T$ is $\omega$-stable.
\end{theorem}

\begin{proof}
Identify the underlying set of $\pg{G}$ and the underlying set of $\mathcal{M}$, so that $\mathcal{M}$ is a representation of $\pg{G}$, and call that set $M$. Since $\pg{G}$ is saturated, it will be enough to show that $S^{\pg{G}}_1(M)$ is countable. We do this by injectively mapping $S^{\pg{G}}_1(M)$ to pairs $(H,a)$ where $H$ is an isomorphism type of an element of $\mathfrak{C}$ and $a\in H$.

Let $p(x)\in S^{\pg{G}}_1(M)$. By Proposition \ref{proposition: model of T can be extended to generic} let $G\succeq \pg{G}$ be generic for $\mathfrak{C}$ such that $G\models p(a)$ for some $a\in G$. Observe that since $\pg{G}$ is elementarily embedded in $G$, any finite set strongly embedded in $\pg{G}$ is also strongly embedded in $G$, i.e., $\pg{G}\sqsubseteq G$. By Corollary \ref{corollary: replace self-sufficient subgraph}, let $\mathcal{N}=(N,S)$ represent $G$ such that $\mathcal{M}\strong \mathcal{N}$.

Assume first that $a\in \cl_{G}(M)$. Then there is some finite independent $X\subseteq M$ such that $a\in \cl_{G}(X)$. Let $Y=\sscl_{\mathcal{N}}(X\cup\set{a})$ and let $H$ be the pregeometry induced on $Y$ by $G$. Then because $X\subseteq Y\cap M$, it must be that $\delta_{\mathcal{N}}(Y/M) \leq 0$, and since $M\strong \mathcal{N}$, we have $\delta_{\mathcal{N}}(Y/M) = 0$, $M\cup Y\strong \mathcal{N}$, and $\mathcal{N}[M\cup Y] = \mathcal{M}\amalg \mathcal{N}[Y]$. So $G_{\mathcal{N}[M\cup Y]} = \pg{G}\amalg_{\pg{G}\cap H} H$. Define $\Theta(p) = (H,a)$.

Observe that by Lemma \ref{lemma: atomic type determines full type for sqsubseteq}, the type $p$ is recoverable from $\Theta(p)$ --- it is the type of the image of $a$ over the image of $M$ in a $\sqsubseteq$-embedded copy of $\pg{G}\amalg_{\pg{G}\cap H} H$ in $\pg{G}$, which by Lemma \ref{lemma: atomic type determines full type for sqsubseteq} is unique. So $\Theta$ is injective on types of $S^{\pg{G}}_1(M)$ whose realizations depend on $M$. But, if $a\notin \cl_{G}(M)$, then $M\cup \set{a}\strong \mathcal{N}$ and $G_{\mathcal{N}[M\cup \set{a}]} = \pg{G}\amalg_{\emptyset} a$. Thus, again by Lemma \ref{lemma: atomic type determines full type for sqsubseteq}, there is a unique type in $S^{\pg{G}}_1(M)$ whose realizations are independent of $M$. All in all, $S^{\pg{G}}_1(M)$ is countable.
\end{proof}

In order to address quantifier elimination for $T$, we only need to be able to speak of finite $\sqsubseteq^*$-extensions.

For each finite pregeometry $H$, letting $\bar{h}$ be an enumeration of the elements of $H$ as an ordered tuple, let $\Phi_H(\bar{x})$ be the full\footnote{While technically the atomic type is not finite, it is isolated by its restriction to the finite sublanguage $\mathcal{L}^{\leq |X|} = \setarg{I_n}{n\leq |X|}$.} atomic type of $\bar{h}$. Let $\Phi^*_H(\bar{x})$ be the formula stating additionally that $\bar{x}\sqsubseteq^* G$ in the ambient structure $G$, namely: ``$\Phi_{H}$ holds, and whenever $X_1, X_2\subseteq \bar{x}$, denoting $n:= \dm(\cl_H(X_1)\cap \cl_H(X_2))+1$, every $y_1,\dots, y_{n}\in \cl_G(X_1)\cap \cl_G(X_2)$ are dependent''. Since $X_i$ is a subtuple of $\bar{x}$, the set $\cl_G(X_i)$ is definable by a quantifier free formula, and since $\bar{x}$ is isomorphic to $H$, the set $\cl_H(X_i)$ is known. Thus, the formula $\Phi^*_H$ is a conjunction of universal formulas, i.e., universal.

\begin{definition}
Define $\LEX$ to be the language $\mathcal{L}$ enriched by a predicate symbol for each formula of the form $\exists \bar{y}\Phi^*_H(\bar{x}\bar{y})$, where $H\in\mathfrak{C}$.
\end{definition}

We interpret a pregeometry as an $\LEX$ structure in the obvious way, implicitly assuming that every $\LEX$-theory forces the ``correct'' interpretation.

\begin{prop}
$T$ has quantifier elimination in the language $\LEX$.
\end{prop}

\begin{proof}
Since $\pg{G}$ is saturated, also its definable expansion to the language $\LEX$ is saturated. Therefore, it is enough to show that the quantifier free $\LEX$-type of a finite tuple $\bar{a}\in \pg{G}$ implies the full type of $\bar{a}$.

Let $\bar{a},\bar{b}\subseteq\pg{G}$ be finite such that they have the same quantifier free $\LEX$-type. Let $A\sqsubseteq^*\pg{G}$ be an extension of $\bar{a}$ such that $\prerank(A)$ is minimal. In particular, $A\strong_r \pg{G}$, so by Corollary \ref{corollary: sqsubseteq equivalence}, $A\sqsubseteq \pg{G}$. Since $\pg{G}\models \exists \bar{y}\Phi^*_A(\bar{a},\bar{y})$, also $\pg{G}\models \exists \bar{y}\Phi^*_A(\bar{b},\bar{y})$. Let $B\sqsubseteq^*\pg{G}$ witness this. Clearly, $\prerank(B)$ has to also be minimal among $\sqsubseteq^*$-embedded extension of $\bar{b}$, so also $B\sqsubseteq \pg{G}$. As $A$ and $B$ are finite, isomorphic, and strongly embedded in $\pg{G}$, by genericity there is an automorphism of $\pg{G}$ extending the isomorphism between $A$ and $B$. In particular, $\tp^{\pg{G}}(\bar{a}) = \tp^{\pg{G}}(\bar{b})$.
\end{proof}

So in the language $\mathcal{L}$, the theory $T$ has quantifier elimination up to boolean combinations of $\LEX$ quantifier free formulas, which in particular are $\exists\forall$ $\mathcal{L}$-formulas. The $\mathcal{L}$-theory $T$ is not model complete in general, but it is with respect to $\sqsubseteq$-embeddings.

\begin{lemma}
\label{lemma: one-to-one correspondence between atomic types and strongly embedded S(T) types}
For each $H\in\mathfrak{C}$ there exists a unique type $p_H\in S(T)$ such that $G\models p_H(\bar{a})$ if and only if $\bar{a}\sqsubseteq G$ and $\bar{a}\cong H$.
\end{lemma}

\begin{proof}
The existence of such a type is clear. Take the type of some strongly embedded copy of $H$ in $\pg{G}$. To see that this type is unique, use saturation of $\pg{G}$ and apply Lemma \ref{lemma: atomic type determines full type for sqsubseteq}.

For the explicit definition of the type, recall T3 of the definition of $T$ and observe that for each $A\in \mathfrak{C}$, we have $G\models \exists\bar{y}\Phi^*_A(\bar{a}, \bar{y})$ precisely when $\bar{a}\sqsubseteq A$. By quantifier elimination, this gives the full type of $\bar{a}$.
\end{proof}

\begin{theorem}
\label{theorem: model completeness of T with respect to strong embeddings}
If $G_1\sqsubseteq G_2$ with $G_1,G_2\models T$, then $G_1\preceq G_2$.
\end{theorem}

\begin{proof}
Let $F\subseteq G_1$. Since we need to show $\tp^{G_1}(F) = \tp^{G_2}(F)$, there is no harm in increasing $F$, so assume $F\sqsubseteq G_1$. By transitivity, also $F\sqsubseteq G_2$. By Lemma \ref{lemma: one-to-one correspondence between atomic types and strongly embedded S(T) types} above, $\tp^{G_1}(F) = p_F = \tp^{G_2}(F)$.
\end{proof}

\subsection{Geometric arity}

Evans and Ferreira \cite{DavidMarcoOne, DavidMarcoTwo} showed that when bounding the arity of the hypergraphs in the amalgamation class $\mathcal{C}$, different associated pregeometries arise. To be precise, whenever $k> n$ are non-negative, the pregeometries $G_{\mathcal{M}_k}$ and $G_{\mathcal{M}_n}$ are not isomorphic, even up to localization in a finite set. The argument hinges on the existence of a self-sufficient edge of maximal arity. We will show that this is the only difference, in the sense that $G_{\mathcal{M}_k}$ is a (geometrically) homogeneous elementary extension of $G_{\mathcal{M}_n}$ realizing the non-isolated type of such an edge.

We define ``arity'' of a flat pregeometry as a purely geometric notion, and show in Proposition \ref{proposition: arity equivalence} that the definition indeed coincides with a definition by the arity of hypergraph representations.
\begin{definition}
\label{definition: arity}
Let $n\in \Nats$ be non-negative. Say that a flat pregeometry $G$ is of \emph{arity at most $n$} and write $\arity(G)\leq n$ if whenever $H\sqsubseteq G$ is finite, then $\flatsum{H}(\Sigma) = \dm(H)$, where $\Sigma$ is the collection of closed sets in $H$ of dimension less than $n$. Define $\arity(G)$, the \emph{arity} of $G$, to be the least $n\in\Nats\cup\set{\omega}$ such that $\arity(G)\leq n$.
\end{definition}

The meaning of $G$ being of arity at most $n$ is that the pregeometry is completely determined by independence of $n$-tuples, or in other words, by its reduct to the language ${\mathcal{L}_k:= \setarg{I_k}{k\leq n}\subseteq \mathcal{L}}$. This reflects in the automorphism group of the pregeometry.

\begin{prop}
If $G$ is flat with $\arity(G)\leq n$, then a bijection $f:G\to G$ is an automorphism of $G$ if and only if $\dm(X) = \dm(f[X])$ whenever $X\in [G]^{\leq n}$
\end{prop}

\begin{proof}
Assume the right hand side. So $f$ preserves the dimension and closedness of every closed set in $G$ of dimension less than $n$. Then for every $\Sigma$, finite collection of closed subsets of dimension less than $n$, we have $\flatsum{G}(\Sigma) = \flatsum{G}(f(\Sigma))$, where $f(\Sigma) = \setarg{f[E]}{E\in \Sigma}$. Since $\arity(G)\leq n$, the dimension of any finite set is given by $\flatsum{G}(\Sigma)$ for such a $\Sigma$, hence $f$ preserves dimension of all tuples. The left to right implication holds by definition.
\end{proof}

\begin{observation}
Let $G$ be a flat pregeometry with $\arity(G)\leq n$. If $H\sqsubseteq G$, then $\arity(H)\leq n$. However, $H\sqsubseteq^* G$ does not imply any finite bound on $\arity(H)$.
\end{observation}

\begin{definition}
We say that a pregeometry $H$ is a \emph{circuit} if $H$ is not independent, but every proper subset of $H$ is independent.
\end{definition}

\begin{prop}
\label{proposition: arity equivalence}
For $G$ a flat pregeometry. The following are equivalent:
\begin{enumerate}
\item
$\arity(G)\leq n$
\item
If $H\sqsubseteq G$ is such that every $X\in [H]^{\leq n}$ is independent, then $H$ is independent
\item
If $H\sqsubseteq G$ is a circuit, then $|H|\leq n$
\item
Whenever $\hg{A}=(M,R)$ is a good representation of $G$, then $R\subseteq [M]^{\leq n}$
\item
There exists $\hg{A}=(M,R)$ a representation of $G$ with $R\subseteq [M]^{\leq n}$
\end{enumerate}
\end{prop}

\begin{proof}
$(1) \Rightarrow (2)$: First assume $H$ is finite. Then, letting $\Sigma$ be the collection of all closed sets in $H$ of dimension less then $n$, in fact $\Sigma = [H]^{<n}$. Thus, the alternating sum $\flatsum{H}(\Sigma)$ is a true inclusion-exclusion, which by $\arity(G)\leq n$ results in $\dm(H)=\flatsum{H}(\Sigma)=|\bigcup \Sigma| = |H|$. Since an infinite set is independent if and only if each of its finite subsets is independent, $(2)$ also holds for $H$ of an arbitrary size.

$(2) \Rightarrow (3)$: If $|H|>n$ and each of its proper subsets is independent, then in particular each $X\in[H]^{\leq n}$ is independent. By $(2)$, $H$ cannot be dependent, so $H$ is not a circuit.

$(3) \Rightarrow (4)$: Let $\hg{A}=(M,R)$ be a good representation of $G$. If $e\in R$, then $e$ is a circuit, and $\delta_{\hg{A}}(e) = |e|-1 = \dm_{\hg{A}}(e)$, so $e\strong \hg{A}$ and $e\sqsubseteq G$. Thus, $|e|\leq n$.

$(4) \Rightarrow (1)$: Let $H\sqsubseteq G$. By Corollary \ref{corollary: sqsubseteq equivalence}, let $\hg{A}=(M,R)$ be a good representation of $G$, such that $X\strong \hg{A}$, where $X$ is the underlying set of $H$. Let $\Sigma$ be as in Definition \ref{definition: arity}. Note $\flatsum{H}(\Sigma) = \flatsum{G}(\Sigma_G)$, because $H\sqsubseteq^* G$. Now apply Lemma \ref{lemma: flatsum covering edges} to $\Sigma_G$, to get $\flatsum{\hg{A}}(\Sigma_G) = \dm_{\hg{A}}(X) = \dm(H)$.

$(4) \Rightarrow (5)$: Immediate.

$(5) \Rightarrow (4)$: Let $\hg{A}=(M,R)$ be a representation of $G$ with $R\subseteq [M]^{\leq n}$ and let $\hg{A}'=(M,R')$ be some good representation of $G$. Let $e\in R'$ be arbitrary and let $X = \sscl_{\hg{A}}(e)$ and $\Sigma = \setarg{\cl_G(X_0)}{X_0\in [X]^{<n}}$. Since $X\strong \hg{A}$, we have $\delta_{\hg{A}}(\cl_G(X_0)) = \dm_G(X_0) = \delta_{\hg{A}}(\cl_G(X_0)\cap X)$ for every $X_0\subseteq X$. Then by (6) of Fact \ref{fact: self-sufficiency properties}, $\bigcup\Sigma\strong \hg{A}$ and $R[\bigcup\Sigma] = \bigcup_{E\in\Sigma} R[E]$. Therefore, by the additional part of Proposition \ref{prop: flatness}, $\flatsum{G}(\Sigma) = \dm_G(\bigcup \Sigma)$. Now with the equality known, applying the additional part of Proposition \ref{prop: flatness} to $\hg{A}'$ yields that $R'[\bigcup\Sigma] = \bigcup_{E\in\Sigma} R'[E]\subseteq [M]^{\leq n}$. In particular, $e\in E$ for some $E$ with $\dm_G(E)< n$. Since $\hg{A}'$ is a good representation, this means $|e| \leq n$.
%
\end{proof}

Towards defining the $n$-ary generic pregeometry, we observe that taking an amalgam does not increase arity.

\begin{cor}
\label{corollary: amalgam for n-ary pregeometries}
For flat pregeometries $H, G_1, G_2$ such that $G_1\amalg_H G_2$ is well-defined, $\arity(G_1\amalg_H G_2) \leq \max\set{\arity(G_1), \arity(G_2)}$.
\end{cor}

\begin{proof}
Use $(4)$ of Proposition \ref{proposition: arity equivalence} in constructing the hypergraph $\hg{B}_1\amalg \hg{B}_2$ of which $G_1\amalg_H G_2$ is the associated pregeometry.
\end{proof}

\begin{definition}
Define $\mathfrak{C}_n$ to be the class of finite flat pregeometries of arity at most $n$. Equivalently, by (5) of Proposition \ref{proposition: arity equivalence}, $\mathfrak{C}_n = \setarg{G_{\hg{A}}}{\hg{A}\in \mathcal{C}_n}$.
\end{definition}

Corollary \ref{corollary: amalgam for n-ary pregeometries} gives that $\mathfrak{C}_n$ is an amalgamation class. Theorem \ref{theorem: Fraisse's theorem - restricted} thus guarantees a unique countable generic structure for $\mathfrak{C}_n$, which we denote $\pg{G}_n$ and call the \emph{generic flat $n$-ary pregeometry}.

\begin{prop}
$G_{\mathcal{M}_n}\cong \pg{G}_n$
\end{prop}

\begin{proof}
By $(5)$ of Proposition \ref{proposition: arity equivalence}, $\arity(G_{\mathcal{M}_n}) \leq n$, so whenever $H\sqsubseteq G_{\mathcal{M}_n}$ is finite, $\arity(H)\leq n$, i.e., $H\in\mathfrak{C}_n$. Now, letting $H\in \mathfrak{C}_n$ be arbitrary, there is some $\hg{A}\in \mathcal{C}_n$ with $H=G_{\hg{A}}$, so by genericity of $\mathcal{M}_n$ we may assume $\hg{A}\strong \mathcal{M}_n$, hence $H=G_{\hg{A}}\sqsubseteq G_{\mathcal{M}_n}$. From both inclusions conclude that ${\setcol{H\sqsubseteq G_{\mathcal{M}_n}}{|H|<\infty} = \mathfrak{C}_n}$.

The rest is exactly the same as in the proof of $\pg{G}\cong G_{\mathcal{M}}$ (Proposition \ref{proposition: G is isomorphic to G_M}), keeping Corollary \ref{corollary: amalgam for n-ary pregeometries} in mind.
\end{proof}

In order to show $\pg{G}_n\models T$, we will go through some technical lemmas showing that for every $k$ and $n\geq 3$, every flat pregeometry $G$ of arbitrary arity is $\sqsubseteq^k$-embedded in some pregeometry $H$ of arity at most $n$.

\begin{definition}
Let $F$ be a flat pregeometry. 
We say that $H\in\mathfrak{C}_n$ is an \emph{$n$-resolution} of $F$ if $F\sqsubseteq^* H$ and there is no $F\subseteq F'\sqsubseteq H$ distinct from $H$. 
\end{definition}

\begin{remark}
If $\arity(F)\leq n$, then $F$ is an $n$-resolution of itself.
\end{remark}

In order to find resolutions of arbitrary pregeometries, we will use resolutions of circuits. The next lemma provides an explicit construction of such resolutions.

\begin{lemma}
\label{lemma: edge resolution}
For $m>n\geq 3$, let $F$ be the unique circuit of size $m$. Then there exists $G_F$, an $n$-resolution of $F$. Moreover, for every natural $p$, $G_F$ can be chosen such that $F\sqsubseteq^p G_F$.
\end{lemma}

\begin{proof}
Denote the underlying set of $F$ by $M$. Enumerate $M=\set{f_1,\dots, f_m}$ and for every natural non-negative number of the form $l=r\cdot m + s$, identify $f_l:= f_s$. Fix some $k> 3(m+1)$ divisible by $3$ such that $f_k\neq f_1$. Let $a_1,\dots, a_k$ be new elements. Denote $b_1 = a_1$, $b_2 = a_{\frac{k}{3}}$, $b_3 = a_{\frac{2k}{3}}$. Define
\begin{align*}
N = &M \cup \setcol{a_i}{1\leq i \leq k}
\\
R = &\setcol{\set{a_i, f_i,\dots, f_{i+(n-3)}, a_{i+1}}}{1\leq i \leq k}
\\
&\cup \set{\set{a_k,f_k,\dots, f_{k+(n-3)},a_1}, \set{b_1, b_2, b_3, f_1,\dots, f_{n-3}}}
\end{align*}
and $\hg{A} = (M, R)$. We show that $G_F:= G_{\hg{A}}\in\mathfrak{C}_n$ is as desired.

\medskip
\noindent\textbf{Claim 1.} For every $X\subseteq N$,
\begin{enumerate}[(i)]
\item
If $M\nsubseteq X$ and $X\nsubseteq M$, then $\delta_{\hg{A}}(X) > |X\cap M|$.
\item
If $M\subset X \subset N$, then $\delta_{\hg{A}}(X) \geq m$.
\item
If $|X|>n$, then $\delta_{\hg{A}}(X) \geq n$.
\end{enumerate}
\begin{proof}[Proof of claim.]
Note that (i), (ii) are true when $|X\setminus M| = 1$, because $R[X]$ is empty. Let $X$ be minimal contradicting either (i) or (ii), in particular $|X\setminus M| > 1$. Then for each $a_i\in X$, we may assume $a_i$ appears in at least two edges in $R[X]$, or else $X\setminus\set{a_i}$ also contradicts either (i) or (ii). Thus, by construction:
\begin{enumerate}
\item
If $1< i<\frac{k}{3}$ and $a_i\in X$, then $a_j\in X$ for each $1\leq j\leq \frac{k}{3}$.
\item
If $\frac{k}{3}<i <\frac{2k}{3}$ and $a_i\in X$, then $a_j\in X$ for each $\frac{k}{3}\leq j \leq \frac{2k}{3}$.
\item
If $\frac{2k}{3}<i \leq k$ and $a_i\in X$, then $a_1\in X$ and $a_j\in X$ for each $\frac{2k}{3}\leq j \leq k$.
\end{enumerate}

Assume for a moment $X$ fails (i). As $\frac{k}{3} > m$, the conclusion of $(1)$ would imply $M\subseteq X$, and similarly for the conclusions of $(2)$ and $(3)$. Then $X\setminus M\subseteq\set{b_1,b_2,b_3}$. However, if $b_1\in X$, then either $a_2$ or $a_k$ is an element of $x$, which cannot be, hence $b_1\notin X$ and similarly $b_2, b_3\notin X$. So $X\subseteq M$, in contradiction. This proves (i).


Now assume $X$ fails (ii). If $\set{b_1, b_2, b_3, f_1,\dots, f_{n-3}}\notin R[X]$, then mapping each $a_i\in X\setminus M$ to $\set{a_i, f_i, \dots, f_{i+(k-2)}, a_{i+1}}$ is a surjection onto a superset of $R[X]$, hence $|X\setminus M| \geq |R[X]|$ and $\delta_{\hg{A}}(X)\geq \delta_{\hg{A}}(M) = m$. So it must be that $b_1, b_2, b_3\in X$. Denote $Y_1 = \set{a_1,\dots, a_{\frac{k}{3}}}$, $Y_2 = \set{a_{\frac{k}{3}},\dots,  a_{\frac{2k}{3}}}$, $Y_3 = \set{a_{\frac{2k}{3}},\dots, a_k, a_1}$. Since each $a_i$ appears in at least two edges in $X$, by (1)-(3) above, at least two of $Y_1, Y_2, Y_3$ are subsets of $X$. Since $X\neq N$, at most two of $Y_1, Y_2, Y_3$ are subsets of $X$, hence $X= Y_i\cup Y_j$ for distinct $i,j\in\set{1,2,3}$. But $\delta_{\hg{A}}(Y_i\cup Y_j) = m$, in contradiction to our choice of $X$. This finishes (ii).

Assume (iii) does not hold for some $X$. If $M\subseteq X$, then $\delta_{\hg{A}}(X) \geq \dm_{\hg{A}}(M)$, and by part (ii) we observe $\dm_{\hg{A}}(M) = \delta_{\hg{A}}(N) =m-1 \geq n$. If $X\subseteq M$, then clearly $\delta_{\hg{A}}(X) = |X|$. So by part (i) it must be that $|X\cap M|< \delta_{\hg{A}}(X) \leq n-1$. By assumption, it must be that $R[X]$ is not empty, so $|X\cap M|= n-2$. If $a_i\in X$ is such that $i\neq 1, \frac{k}{3},\frac{2k}{3}$, then $a_i$ appears in at most one edge in $R[X]$, i.e., $\delta_{\hg{A}}(X\setminus \set{a_i}) \leq \delta_{\hg{A}}(X)$. Then for any $X'\subseteq X$ such that $X'\cap (M\cup \set{b_1,b_2,b_3}) = X\cap (M\cup \set{b_1,b_2,b_3})$, we have $\delta_{\hg{A}}(X') \leq \delta_{\hg{A}}(X)<n$. In particular, we my assume $|X|=n+1$. Since there are no two edges $e_1,e_2\in R[X]$ such that $|e_1\cap e_2| \geq n-1$, clearly $|R[X]|\leq 1$ and $\delta_{\hg{A}}(X)\geq n$, in contradiction.
\end{proof}

Part (i) of Claim 1 immediately yields that for every $M_0\subseteq M$, either $M_0$ is closed in $\hg{A}$, in which case $\dm_{\hg{A}}(M_0) = |M_0|$, or $M\subseteq\cl_{\hg{A}}(M_0)$. Consequently, $F\sqsubseteq^* G_{F}$.

Part (iii) of Claim 1 shows that the non-trivial $(n-1)$-dimensional closed sets in $\hg{A}$ are precisely the edges. Thus, $\hg{A}$ is the unique good representation of $G_{\hg{A}}$. By Corollary \ref{corollary: sqsubseteq equivalence}, this means there exists $F\subseteq F'\sqsubseteq G_{\hg{A}}$ distinct from $G_{\hg{A}}$ if and only if there exists some $M\subseteq M'\strong \hg{A}$ distinct from $N$. But $\delta(\hg{A}) = m-1$, so (ii) guarantees no such $M'$ exists. Then $G_F$ is an $n$-resolution of $F$.

We address the additional part of the statement.

\medskip
\noindent\textbf{Claim 2.} If $X\subseteq N$ with $\dm_{\hg{A}}(X) < m-1$, then $|X|\leq (m+2)m$.
\begin{proof}[Proof of claim]
When writing $a_i$ for some $i>k$, we mean $a_{i-k}$. Without loss of generality, by increasing $X$ assume $\delta_{\hg{A}}(X) = \dm_{\hg{A}}(X)$.

Since $|M|=m$, we may assume $X\nsubseteq M$. Let $C=\set{a_i,\dots, a_{i+l}}$ be maximal such that for each $j< l$, the elements $a_{i+j}$ and $a_{i+j+1}$ appear together in an edge in $R[X]$. If $l\geq m$, then each element of $M$ appears in some edge in $R[X]$, so $M\subseteq X$ in contradiction to $\dm(X)< m-1$. Therefore $|C|<m$. 

Assume $\set{b_1, b_2, b_3, f_1,\dots, f_{n-3}}\notin R[X]$. Then ${\delta_{\hg{A}}(C/(X\setminus C)) = 1}$. Repeating this process for $X\setminus C$ in the role of $X$, stripping away sets of the form of $C$ until we are only left with elements of $M$, we have removed at most $\delta_{\hg{A}}(X)<m-1$ many sets of size at most $m$ and are left with at most $m$ elements, so $|X|\leq (m-2)m+m$.

Now, if $\set{b_1, b_2, b_3, f_1,\dots, f_{n-3}}\in R[X]$, for $i\in\set{1,2,3}$ let $C_i$ containing $b_i$ be of the form of $C$, and denote $D = C_1\cup C_2\cup C_3$. Then $\delta_{\hg{A}}(D/(X\setminus D)) = 2$, and as we've seen $|X\setminus D|\leq (m-2)m+m$. Then $|X|\leq 3m + (m-2)m + m = (m+2)m$.
\end{proof}

Fix some natural $p$. We show that if $k$ is large enough, then $F\sqsubseteq^p G_F$. Choose $k$ to be large enough so that, by Claim 2, $G_F$ is not a union of $2^{(|F|+p)}$ closed sets of dimension less than $m-1$.

Assume for a contradiction that $F\not\sqsubseteq Z\sqsubseteq^* G_F$, $|Z\setminus F|\leq p$ and let $\Sigma$ be a collection of closed sets in $Z$ witnessing this. Note that $|\Sigma| \leq 2^{|Z|}$ and every $E\in \Sigma$ has $\dm_{\hg{A}}(E)<m-1$, or otherwise $E=Z$. Recall that $M$ is the underlying set of $F$ and that every proper subset of $F$ is independent. Then
\[
|M\cap \bigcup\Sigma|=\flatsum{F}(\Sigma_F) >\flatsum{Z}(\Sigma) \geq \dm_{\hg{A}}(\bigcup \Sigma) \geq \dm_{\hg{A}}(M\cap \bigcup \Sigma).
\]
Since $M\cap \bigcup \Sigma$ is not independent, it must be that $M\subseteq\bigcup \Sigma$ and $\flatsum{Z}(\Sigma) = m-1$. By $Z\sqsubseteq^* G_F$, we have that $\flatsum{G_F}(\Sigma_{G_F}) = \flatsum{Z}(\Sigma) = m-1 = \dm(\bigcup \Sigma_{G_F})$. By the additional part of Proposition \ref{prop: flatness}, we get that $\bigcup \Sigma_{G_F} \strong \hg{A}$. Since $G_F$ is an $n$-resolution of $F$, this means $\bigcup \Sigma_{G_F} = N$. By choice of $k$, this cannot be. Conclude $F\sqsubseteq^p G_F$.
\end{proof}

\begin{digression*}
Resolutions motivate Morley rank in $\pg{G}$ and its similarity to $\prerank$. We can see inductively that an $n$-sized circuit has Morley rank $\geq\omega^{n-4}$ over any $n-1$ of its elements.

For $n>3$, an $n$-sized circuit can resolve in infinitely many mutually exclusive ways into a finite configuration of $n-1$-sized circuits. For a configuration given by $k$ $n$-sized circuits, each one can resolve independently of the others, which by the induction hypothesis gives the configuration Morley rank at least $k\cdot \omega^{n-5}$. Since $k$ is unbounded, the Morley rank of the $n$-sized circuit limits to at least $\omega^{n-4}$. As a base for the induction, it is enough to note that a $4$-sized circuit has infinitely many distinct $3$-resolutions, i.e., Morley rank at least $1 = \omega^0$.

This means that the Morley rank of $T$ is at least $\omega^\omega$, since the type of a point $b$ independent from (a basis of) $\pg{G}$ is, for example, the limit of increasingly large circuits involving $b$ and elements (of a basis of) $\pg{G}$.

Similarly, one can show inductively that the Morley rank of any one specific $n-1$-resolution of an $n$-sized circuit over $n-1$ of its elements is bounded from above by $\omega^{n-3}$.

Thus, the Morley rank of $T$ is precisely $\omega^\omega$, the rank of a point $b$ independent from a basis of $\pg{G}$ --- the unique generic type over $\pg{G}$. This also gives us the novel result that the type of an element $b$ does not fork over a set $B$ if and only if $b$ is independent from $B$, i.e, $\pg{G}$ is isomorphic to its own forking geometry!
\end{digression*}

In an $n$-resolution of $F$ with some $F_0\sqsubseteq F$, if we want to preserve strong embeddedness of $F_0$, we cannot require $H\in\mathfrak{C}_n$, because $F_0$ may have strongly embedded circuits of size greater than $n$. This issue arises when wishing to replace $F$ in the amalgam $F\amalg_{F_0} H$ with $F'$, an $n$-resolution of $F$, for the sake of lower arity --- if $F_0$ is no longer strongly embedded in $F'$, the amalgam $F'\amalg_{F_0} H$ does not necessarily exist. There will be no harm in leaving $F_0$ ``unresolved'' in $F'$, because what seems like a circuit in $F$, may in fact be a part of a low-arity configuration in $H$.

To proceed, we generalize the notion of arity and resolution to make sense \emph{over} some strongly embedded subpregeometry.

\begin{definition}
Say that a flat pregeometry $G$ is of arity at most $n$ \emph{over $F\sqsubseteq G$} if there exist good representations $(M_0, R_0)\strong (M, R)$ of $F\sqsubseteq G$ with $R\setminus R_0 \subseteq [M]^{\leq n}$. Write this statement as $\arity(G/F)\leq n$ and let $\arity(G/F)$ be the least $n$ such that $\arity(G/F)\leq n$.
%
\end{definition}

\begin{observation}
\label{observation: arity of one-sided amalgam with resolution over base}
In the vein of Corollary \ref{corollary: amalgam for n-ary pregeometries}, for flat pregeometries $H, G_1, G_2$ such that $G_1\amalg_H G_2$ is well-defined, $\arity(G_1\amalg_H G_2) \leq \max\set{\arity(G_1/H), \arity(G_2)}$.
\end{observation}

\begin{definition}
For $F_0\sqsubseteq F$, say that $H\in\mathfrak{C}$ is an $n$-resolution of $F$ \emph{over} $F_0$ if $\arity(H/F_0) \leq n$, $F\sqsubseteq^* H$, there is no $F\subseteq F'\sqsubseteq H$ distinct from $H$, and $F_0\sqsubseteq H$.
\end{definition}

\begin{observation}
For a flat pregeometry $F$, $\arity(F) = \arity(F/\emptyset)$ and $H$ is an $n$-resolution of $F$ if and only if $H$ is an $n$-resolution of $F$ over $\emptyset \sqsubseteq F$.
\end{observation}

We now show that an $n$-resolution over a strongly embedded subpregeometry always exists. In particular, the observation immediately above implies that every flat pregeometry has an $n$-resolution.

\begin{prop}
\label{proposition: n-resolution}
For every $F\in \mathfrak{C}$, $F_0\sqsubseteq F$ and $n\in\Nats$, there exists $H\in \mathfrak{C}$, an $n$-resolution of $F$ over $F_0$. Moreover, for every natural $p$, $H$ can be chosen so that $F\sqsubseteq^p H$.
\end{prop}

\begin{proof}
Let $\hg{A}=(M,R)$, $\hg{A}_0 = (M_0, R_0)$ be a good representations of $F, F_0$ respectively, such that $\hg{A}_0\strong \hg{A}$. For each $e\in R\setminus R_0$, seen as a subpregeometry of $F$, using \ref{lemma: edge resolution}, let $G_e$ be an $n$-resolution of $e$. If $|e|\leq n$, take $G_e = e$. Enumerate $R\setminus R_0 =\set{e_1,\dots, e_k}$, define $H_1 = F$ and inductively define $H_{i+1} = H_i\amalg_{e_i} G_{e_i}$. Denote $H=H_{k+1}$.

Recalling the explicit construction of an amalgam of pregeometries (definitions \ref{definition: hypergraph amalg}, \ref{definition: geometric amalgam}), at every stage the pregeometry $H_{i+1}$ is represented by a hypergraph $\hg{D}_{i+1}= (N, S_{i+1})$ which is $\hg{D}_i\amalg \hg{B}_i$, where $\hg{B}_i$ is a (good) representation of $G_{e_i}$. Then for every $j\geq i$ we have that still $e_j\in S_i[M]$, and in particular $\delta_{\hg{D}_i}(e_j) = |e_j|-1 = \dm_{F}(e_j)=\dm_{\hg{D}_i}(e_j)$. Therefore, $e_j\strong H_i$, implying $e_j\sqsubseteq H_i$. This means that $H_{i+1}$ is well defined at every step. 

Unraveling the construction, we see that the order in which we enumerate the edges makes no difference to the resulting hypergraph $\hg{D}_{k+1}$ and associated pregeometry $H$. Thus, given any $e\in R\setminus R_0$, we may re-enumerate so that $e=e_k$ and get $H = H_k\amalg_e G_e$.

Now we check that $H$ is an $n$-resolution of $F$ over $F_0$. Clearly, $H\in \mathfrak{C}$.
\begin{itemize}
\item
By Lemma \ref{lemma: amalg properties}, since $e_i\sqsubseteq^* G_{e_i}$, we have $H_i\sqsubseteq^* H_{i+1}$, so inductively $F\sqsubseteq^* H$.
\item
By construction, inductively, for each $i$ the restriction of $H_i$ to $M_0$ is $F_0$. Since $\hg{D}_{k+1}[M_0] = \hg{A}_0$, we have $G_{\hg{D}_{k+1}[M_0]} = G_{\hg{A}_0} = F_0$. By (2) of Fact \ref{fact: self-sufficiency properties}, this implies $\hg{A}_0\strong \hg{D}_{k+1}$ and consequently $F_0\sqsubseteq H$.
\item
For each edge in $\hg{D}_{k+1}$, its dimension remains unchanged from the stage when it was introduced into the construction. By choosing only good representations of the pregeometries $G_{e}$ during construction, the resulting $\hg{D}_{k+1}$ is a good representation of $H$. In the previous item we saw $\hg{A}_0\strong \hg{D}_{k+1}$, so $\arity(H/F_0) \leq n$.
\item
Lastly, let $F\subseteq F'\sqsubseteq H$, then by (1) of Corollary \ref{corollary: strong in amalgam means strong in G_2}, $F'\cap G_e\sqsubseteq G_e$ for every edge $e$. But $G_e$ is an $n$-resolution of $e$ so it must be that $G_e\subseteq F'$. Then $F' = H$. 
\end{itemize}
We prove the additional part. Fix $p$ and assume we had chosen all the $G_e$ such that $e\sqsubseteq^p G_e$. Let $Z\sqsubseteq^* H$ contain $F$ such that $|Z\setminus F|\leq p$. For each $e$, by (2) of Corollary \ref{corollary: strong in amalgam means strong in G_2}, since $e\sqsubseteq^* G_e$, we have $Z\cap G_e\sqsubseteq^* G_e$. By $e\sqsubseteq^p G_e$, in fact $e\sqsubseteq Z\cap G_e$. For each $e$ denote $Z_e = Z\cap G_e$. Applying Observation \ref{observation: intermediate amalgam} iteratively, we see that the restriction of $H$ to $Z$ is the amalgam of the pregeometries $F\amalg_e Z_e$ over $F$. By the last item of Lemma \ref{lemma: amalg properties}, for each $e$ it holds that $F\sqsubseteq F\amalg_e Z_e$, so inductively we get $F\sqsubseteq Z$.
\end{proof}

We can finally determine the theory of $\pg{G}_n$. The proof below is similar to that of $\pg{G}\models T$ (Proposition \ref{proposition: T is satisfied by pg{G}}), but does require more consideration.

\begin{prop}
There exists some $\tau$ such that $\pg{G}_n\models T_{\tau}$ (Recall Definition \ref{definition: theory of generic pregeometry}). In particular $\pg{G}_n\models T$.
\end{prop}

\begin{proof}
Both T1 and T2 are clear. We show T3 holds. Suppose $F\sqsubseteq^* \pg{G}_n$, $F\sqsubseteq H\in \mathfrak{C}$ and fix some natural $k$. We want to find a sufficient condition on $\tau(k)$ so that if $F\sqsubseteq^{\tau(|H|+k)} \pg{G}_n$, then there is an embedding $f:H\to \pg{G}_n$ such that $f[H]\sqsubseteq^k \pg{G}_n$.

By Proposition \ref{proposition: n-resolution}, let $G_0$ be an $n$-resolution of $H$ over $F$ such that $H\sqsubseteq^p G_0$ for some $p>h(h(|H|+k))$, where $h$ is a function as in \ref{lemma: not sqsubseteq is witnessed inside any sqsubseteq superstructure}. Choose some finite $E\sqsubseteq \pg{G}_n$ containing $F$, and let $G = G_0\amalg_{F} E$. Observation \ref{observation: arity of one-sided amalgam with resolution over base} gives $\arity(G)\leq n$, so by genericity we may assume $G$ is strongly embedded into $\pg{G}_n$ over $E$.

Assume that $H\not\sqsubseteq^k \pg{G}_n$ and let $Z\sqsubseteq^* \pg{G}_n$ contain $H$ such that $|Z\setminus H|\leq k$ and $H\not\sqsubseteq Z$. Then by Lemma \ref{lemma: not sqsubseteq is witnessed inside any sqsubseteq superstructure} there is some $V\sqsubseteq^* G$ containing $H$ such that $|V|\leq h(|H|+k)$ and $H\not\sqsubseteq V$. By (1) of Corollary \ref{corollary: strong in amalgam means strong in G_2}, $V_0:= V\cap E\sqsubseteq^* E$. Observation \ref{observation: intermediate amalgam} now implies $G':= G_0\amalg_F V_0\sqsubseteq^* G$. Note that $p>h(|V|)$.

Assume for a moment $F\sqsubseteq V_0$. Then $G_0\sqsubseteq G'$ and, considering ${H\not\sqsubseteq V\sqsubseteq^* G'}$, an application of Lemma \ref{lemma: not sqsubseteq is witnessed inside any sqsubseteq superstructure} yields some set $V'\sqsubseteq^* G_0$ containing $H$ such that $H\not\sqsubseteq V'$. But $|V'|\leq h(|V|) < p$, in contradiction to $H\sqsubseteq^p G_0$. So it must be that $F\not\sqsubseteq V_0$, hence $F\not\sqsubseteq^{h(|H|+k)}\pg{G}_n$. As before, setting $\tau$ greater or equal to the $h$ function of Lemma \ref{lemma: not sqsubseteq is witnessed inside any sqsubseteq superstructure} gives us what we want.
\end{proof}

While tempting to call $T$ the theory of generic flat pregeometries, there are still more generic flat pregeometries -- in the sense that they are generic structures for amalgamation classes -- not sharing this theory. Like in the case of the strongly minimal Hrushovski construction, one can enforce finite multiplicities on certain configurations -- see \cite[Lemmas 17,18]{Hns} for good representations of such. Also, there are models of $T$ besides those that we've seen --- an elementary extension of $\pg{G}_3$ realizing a unique strongly-embedded ``4-ary'' circuit, say. The study of these variants is left for a different paper.

To conclude our investigation, we show that the generic flat pregeometries we have studied form an elementary chain. Fix until the end of this section copies of $\mathcal{M}$ and $\pg{G}$ such that $\pg{G} = G_{\mathcal{M}}$ and $\mathcal{M} = (M,R)$. Also fix $I\subseteq M$, a basis for $\pg{G}$.

\begin{observation}
Let $\hg{A}_1=(M, R_1)$, $\hg{A}_2= (M, R_2)$ be good representations of the same pregeometry and let $\hg{B}_i = (M, R_i\cap [M]^{\leq n})$. Then $G_{\hg{B}_1} = G_{\hg{B}_2}$.

Being good representations, clearly $G_{\hg{B}_1}$ and $G_{\hg{B}_2}$ have the same closed sets of dimension less than $n$. By definition of $\arity(\hg{B}_i) \leq n$, this uniquely determines the entire pregeometry associated to $\hg{B}_i$.
\end{observation}

\begin{definition}
Let $G$ be a flat pregeometry, let $\hg{A}=(M,R)$ be a good representation of $G$ and let $\hg{A}^{\leq n} = (M,R\cap [M]^{\leq n})$. For every $X\subseteq M$ define $\cl_G^n(X) = \cl_{\hg{A}^{\leq n}}(X)$, the \emph{$n$-ary closure of $X$ in $G$}.
\end{definition}

The structure $\mathcal{M}_n$ can be identified in $\mathcal{M}$ as the $n$-ary closure of an infinite independent set.

\begin{lemma}
Let $N = \cl_{\pg{G}}^n(I)$, for some $n\geq 3$. Then $N\strong \mathcal{M}$ and $\mathcal{M}[N] \cong \mathcal{M}_n$.
\end{lemma}

\begin{proof}
For every finite $X\subfin N$, there is $Y\subfin N$ containing $X$ such that $\delta_{\mathcal{M}}(Y/Y\cap I)\leq 0$. To be precise, $|Y\setminus I| \leq |R[Y]\cap [M]^{\leq n}|$. As $Y\cap I\strong \mathcal{M}$, we have $Y\strong \mathcal{M}$ and $R[Y] = R[Y]\cap [M]^{\leq n}$. By $Y\strong \mathcal{M}$, we get that $\dm_{\mathcal{M}[N]}(X) = \dm_{\mathcal{M}[Y]}(X) = \dm_{\mathcal{M}}(X)$. Therefore, since $X$ was arbitrary, $N\strong \mathcal{M}$ and $R[N] = R[N]\cap [M]^{\leq n}$.

Now, we only need to show extension -- $(*)$ of Theorem \ref{theorem: Fraisse's theorem - restricted} -- to get that $\mathcal{M}[N]$ is generic for $\mathcal{C}_n$. Let $\hg{A}\strong \mathcal{M}[N]\strong \mathcal{M}$ and let $\hg{A}\strong \hg{B}\in \mathcal{C}_n$. Increasing $\hg{A}$ by mapping points in $\hg{B}$ independent from $\hg{A}$ in $\hg{B}$ to elements of $I$ that are independent from $\hg{A}$ in $\mathcal{M}$, we may assume $\delta(\hg{A}) = \delta(\hg{B})$. By genericity of $\mathcal{M}$, we may assume $\hg{B}$ is strongly embedded into $\mathcal{M}$ over $\hg{A}$. But then the embedding is clearly into $\mathcal{M}[N]$
\end{proof}

\begin{theorem}
$\pg{G}_3\prec \pg{G}_4\prec \pg{G}_5\prec \dots \prec \pg{G}$
\end{theorem}

\begin{proof}
By the above lemma, identifying $\mathcal{M}_n$ with $\cl^n_{\pg{G}}(I)$, we get
\[
\mathcal{M}_3 \strong \mathcal{M}_4 \strong \mathcal{M}_5 \strong \dots \strong \mathcal{M}
\]
which immediately gives
\[
\pg{G}_3\sqsubseteq \pg{G}_4\sqsubseteq \pg{G}_5\sqsubseteq \dots \sqsubseteq \pg{G}.
\]
Since $\pg{G}_n\models T$ for every $n\geq 3$, by Theorem \ref{theorem: model completeness of T with respect to strong embeddings} we are done.
\end{proof}

\section*{Acknowledgments}
The author would like to thank Uri Andrews and Assaf Hasson for discussing the material and commenting on earlier versions of this paper.

\bibliographystyle{alpha}
\bibliography{../../myrefs}
\end{document}